\documentclass[reqno, 10pt]{amsart}
\usepackage{amsmath, amssymb, amscd, amsthm, amsfonts, amstext,
  amsbsy,  mathrsfs, hyperref,  upgreek, mathtools, stmaryrd, enumitem, bbm, mathabx}
\usepackage[german,polish,english]{babel}
\usepackage[textsize=footnotesize,color=yellow, bordercolor=white]{todonotes}

\usepackage
[a4paper, margin=1.35in]{geometry}

\usepackage[utf8]{inputenc}
\usepackage[T1]{fontenc}

\usepackage{pdflscape} 
\usepackage{afterpage}
\usepackage{etoolbox}

\usepackage{tikz}
\usetikzlibrary{tikzmark}

\newcommand{\RR}{\mathbb{R}}
\newcommand{\NN}{\mathbb{N}}

\newcommand{\EE}{\mathbb{E}}

\newcommand{\TT}{\mathbb{T}}

\newcommand{\Sacks}{\mathbb{S}}
\newcommand{\Silver}[1][{}]{\mathbb{V}_{ #1 }}
\newcommand{\Cohen}{\mathbb{C}}
\newcommand{\Laver}[1][{}]{\mathbb{L}_{ #1 }}
\newcommand{\Hechler}[1][{}]{\mathbb{H}_{ #1 }}
\newcommand{\Eventual}[1][{}]{\mathbb{E}_{ #1 }}
\newcommand{\Miller}{\mathbb{M}}
\newcommand{\Mathias}[1][{}]{\mathbb{R}_{ #1 }}
\newcommand{\Random}{\mathbb{B}}

\newcommand{\PP}{\mathbb{P}}
\newcommand{\QQ}{\mathbb{Q}}
\newcommand{\CC}{\mathbb{C}}
\newcommand{\uBc}{\Lambda_{\mathrm{uB}}}
\newcommand{\Add}{\mathrm{Add}}

\DeclareFontFamily{OT1}{pzc}{}
\DeclareFontShape{OT1}{pzc}{m}{it}{<-> s * [1.100] pzcmi7t}{}
\DeclareMathAlphabet{\mathscr}{OT1}{pzc}{m}{it}

\newcommand{\cN}{\mathrm{N}}
\newcommand{\cI}{I}
\newcommand{\cIs}{I^\star} 

\newcommand{\Ord}{\mathrm{Ord}}

\newcommand{\I}{I} 
\newcommand{\bK}{\mathbf{K}}
\newcommand{\bL}{\mathbf{L}} 

\newcommand{\pow}{\mathscr{P}}

\newcommand{\stem}{\operatorname{stem}}
\newcommand{\spl}{\operatorname{split}}

\newcommand{\res}{\upharpoonright}

\newcommand{\AD}{{\sf AD}}
\newcommand{\ZFC}{{\sf ZFC}}
\newcommand{\ZF}{{\sf ZF}}

\newcommand{\PD}{{\sf PD}}
\newcommand{\CH}{{\sf CH}}

\newcommand{\supp}{\mathrm{supp}}

\DeclareMathOperator{\dom}{dom}
\DeclareMathOperator{\ran}{ran}

\newcommand{\st}{\mid}

\newcommand{\baseset}{\Omega} 
\newcommand{\comm}[1]{}

\newcommand{\Lev}{\mathrm{Lev}} 

\newcommand{\shift}{\sigma}

\newcommand{\np}{}

\newcommand{\XX}{X}

\newcounter{cl}
\AtBeginEnvironment{theorem}{\setcounter{cl}{0}}
\AtBeginEnvironment{lemma}{\setcounter{cl}{0}}
\newtheorem{theorem}{Theorem}[section]
\newtheorem*{theorem*}{Theorem} 
\newtheorem{lemma}[theorem]{Lemma}
\newtheorem{corollary}[theorem]{Corollary}
\newtheorem{proposition}[theorem]{Proposition}

\newtheorem{conditions}[theorem]{Conditions}
\newtheorem{convention}[theorem]{Convention}
\newtheorem{notation}[theorem]{Notation}
\newtheorem{question}[theorem]{Question}

\theoremstyle{definition}
\newtheorem{claim}[theorem]{Claim}
\newtheorem*{claim*}{Claim}
\newtheorem{subclaim}[theorem]{Subclaim}
\newtheorem*{subclaim*}{Subclaim}
\newtheorem{definition}[theorem]{Definition}

\newtheorem{example}[theorem]{Example}
\theoremstyle{remark}
\newtheorem{remark}[theorem]{Remark}

\newtheorem*{case*}{Case}

\newtheorem*{subcase*}{Subcase}

\DeclareMathOperator{\Borel}{\textsc{\rmfamily Borel}}
\newcommand{\Borelcodes}{\mathcal{B}}
\newcommand{\Borelset}{B}
\DeclareMathOperator{\Proj}{\textsc{\rmfamily Proj}}

\DeclareMathOperator{\Meas}{\textsc{\rmfamily Meas}}

\newenvironment{enumerate-(a)}{\begin{enumerate}[label={\upshape (\alph*)}, leftmargin=2pc]}{\end{enumerate}}

\newenvironment{enumerate-(a)-r}{\begin{enumerate}[label={\upshape (\alph*)}, leftmargin=2pc,resume]}{\end{enumerate}}

\newenvironment{enumerate-(A)}{\begin{enumerate}[label={\upshape (\Alph*)}, leftmargin=2pc]}{\end{enumerate}}

\newenvironment{enumerate-(A)-r}{\begin{enumerate}[label={\upshape (\Alph*)}, leftmargin=2pc,resume]}{\end{enumerate}}

\newenvironment{enumerate-(i)}{\begin{enumerate}[label={\upshape (\roman*)}, leftmargin=2pc]}{\end{enumerate}}

\newenvironment{enumerate-(i)-r}{\begin{enumerate}[label={\upshape (\roman*)}, leftmargin=2pc,resume]}{\end{enumerate}}

\newenvironment{enumerate-(I)}{\begin{enumerate}[label={\upshape (\Roman*)}, leftmargin=2pc]}{\end{enumerate}}

\newenvironment{enumerate-(I)-r}{\begin{enumerate}[label={\upshape (\Roman*)}, leftmargin=2pc,resume]}{\end{enumerate}}

\newenvironment{enumerate-(1)}{\begin{enumerate}[label={\upshape (\arabic*)}, leftmargin=2pc]}{\end{enumerate}}

\newenvironment{enumerate-(1)-r}{\begin{enumerate}[label={\upshape (\arabic*)}, leftmargin=2pc,resume]}{\end{enumerate}}

\begin{document}

\title{
Lebesgue's density theorem and definable selectors for ideals} 
\date{\today} 

\author{Sandra M\"uller} 
\address{Sandra M\"uller, 
 Institute of Mathematics, University of Vienna, UZA 1---Geb\"aude 2, Augasse 2--6, 1090 Wien, Austria} 
\email{mueller.sandra@univie.ac.at}


\author{Philipp Schlicht} 
\address{Philipp Schlicht, School of Mathematics, University of Bristol, Fry Building, Woodland Road, Bristol, BS8 1UG, UK} 
\email{philipp.schlicht@bristol.ac.uk} \urladdr{}

\author{David Schrittesser} \address{David Schrittesser, 
 Institute of Mathematics, University of Vienna, UZA 1---Geb\"aude 2, Augasse 2--6, 1090 Wien, Austria} \email{david.schrittesser@univie.ac.at}

\author{Thilo Weinert}
\address{Thilo Weinert, Department of Mathematics
National University of Singapore
Blk S17 Level 4
10 Lower Kent Ridge Road
119076 Singapore} \email{thilo.weinert@univie.ac.at}



\begin{abstract} 
We introduce a notion of density point and prove results analogous to Lebesgue's density theorem for various well-known ideals
on Cantor space and Baire space.
In fact, we isolate a class of ideals for which our results hold. 

In contrast to these results, 
we show that there is no reasonably definable selector 
that chooses representatives 
for the equivalence relation on the Borel sets of having countable symmetric difference. 
In other words, there is no notion of density which makes the ideal of countable sets satisfy an analogue to the density theorem. 

The proofs of the positive results use only elementary combinatorics of trees, while 
the negative results rely on forcing arguments. 


\end{abstract} 

\maketitle











\np

\section{Introduction}\label{s.intro}

Given a $\sigma$-ideal $\I$ on a Polish space $X$, 
consider the equivalence relation $=_\I$ on $\pow(X)$ defined by 
$A =_\I B \iff A \triangle B \in \I$.
A map $D$ with $\dom(D),\ran(D) \subseteq \pow(X)$ 
is called a \emph{
selector for $=_\I$} if 
$D(A)=_\I A$ (i.e., $D$ chooses a representative from the $=_\I$-equivalence class of $A$) and $A =_\I B \Rightarrow D(A)=D(B)$ (i.e., $D$ is invariant with respect to $=_\I$) for all $A, B \in \dom(D)$. 
While a selector for $=_\I$ on $\pow(X)$ can always be found using the axiom of choice, it is natural to ask: For which $\sigma$-ideals is there a selector $D$ with domain the Borel subsets of $\XX$ that is  \emph{definable} in some given sense?
For instance, we are interested in Borel measurable and projective selectors. 

To make this more precise, let us consider the ideal of Lebesgue null sets on $\RR$ and the meager ideal on any Polish space as examples.
First we consider the equivalence relation $=_\mu$ on the collection $\Meas$ of Lebesgue measurable subsets of $\RR$, given by equality up to a Lebesgue null set. 
By Lebesgue's density theorem, we obtain a selector by assigning to each measurable set $A\subseteq \RR$ the set $D_\mu(A)$ of points of density $1$, as follows: Recall that 
\[
d^\mu_A(x)=\liminf_{\epsilon\rightarrow0} \frac{\mu\big(B_\epsilon(x)\cap A\big)}{\mu\big(B_\epsilon(x)\big)}
\]
and $D_\mu(A) = \{x \in \RR \mid d^\mu_A(x) = 1\}$\label{density1} for $A \in \Meas$.\footnote{See also
  \cite{andrettacamerloconstantini} and
  \cite{andrettacamerlo} for recent results on the descriptive set theory of Lebesgue's density theorem.} Note that in measure theory $D_\mu(A)$ is usually denoted by $\Phi(A)$.
\begin{theorem}[Lebesgue's density theorem]
For any $A \in \Meas$, $A =_\mu D_\mu(A)$.
\end{theorem}
It follows that $A \mapsto D_\mu(A)$ is a selector with domain $\Meas$.
To gauge the definability of this selector, take recourse to a standard coding of Borel sets (see e.g. \cite[35.B]{Ke95}): 
Fix a $\Pi^1_1$ set $\Borelcodes \subseteq \pow(\omega)$ of Borel codes and $\Sigma^1_1$ sets $P, Q \subseteq \pow(\omega) \times \RR$ such that $\{y\in\RR \mid P(x,y)\} = \{ y \in \RR \mid \neg Q(x,y)\}$ for all $x \in \Borelcodes$ and  every Borel set is of this form.
Write $\Borelset_x$ for this set, the Borel set \emph{coded by $x$}. 
With such a coding, standard arguments  on the complexity of measure (see \cite[Theorem 2.2.3]{kechris:measure}) 
show that $D_\mu$ restricted to the Borel sets is induced by a definable map on the Borel codes as follows: 
There is a map $D\colon \Borelcodes \to \Borelcodes$  such that
\[
\Borelset_{D(x)} = D_\mu(\Borelset_x)
\]
for all Borel codes $x \in \Borelcodes$, 
and the graph of $D$ is a Boolean combination of $\Sigma^1_1$ sets. In particular, $D$ is universally Baire measurable.\footnote{We refer the reader to \cite{Ke95} for standard definitions in descriptive set theory throughout this paper. We also review the definition of universally Baire in Section~\ref{section complexity}.}

The situation is entirely analogous for the Lebesgue measure (i.e. coin-tossing measure) on Cantor space ${}^\omega 2$ (see e.g. 
\cite[Section 8]{andrettacamerlo2012} or \cite[Proposition 2.10]{Mi08}). 
It further follows from the isomorphism theorem for measures \cite[Theorem 17.41]{Ke95} that a definable selector exists for the null ideal with respect to any Borel probability measure on a standard Borel space.

Similarly, a definable selector (with domain the Borel sets) exists for the \emph{meager ideal} on any Polish space. This is the smallest $\sigma$-ideal containing all 
nowhere dense sets (i.e., sets whose closure has empty interior). 
When $I$ is the meager ideal we can define a selector by letting
\[
D_\I(A)=\bigcup\{U\mid U\text{ open, $A\cap U$ comeager in $U\}$}
\] 
for any set $A$ with the Baire property (i.e., any set which has meager symmetric difference with some Borel set).\footnote{For a different definition of density point which generalizes the standard notion and is at the same time meaningful for the meager ideal, see \cite{Poreda85} (we discuss this in Section~\ref{section convergence density}).}

\medskip

In this article we isolate a class of ideals, the \emph{strongly linked tree ideals} for which likewise there exists a definable selector (cf.\ Definition~\ref{definition strongly linked}).
For this result, we utilize the connection between $\sigma$-ideals and \emph{forcing}.\footnote{This connection has been explored for instance in \cite{brendlekhomskiiwohofsky, goldstern1995tree, kanovei2013canonical, ikegami2010, shelahspinas, MR2391923}.} To make our paper as accessible as possible we give a short but self-contained review of this connection in Section~\ref{section tree ideals}.

We shall restrict our attention to spaces $X$ of the form ${}^\omega \baseset$ with the product topology, where $\baseset=\{0,1\}$ or $\baseset=\omega$, i.e., Cantor and Baire space.\footnote{Note that for any $\sigma$-ideal $\I$ that contains all singletons on a Polish space $\XX$, there is a set $A \in \I$ such that 
$\XX\setminus A$ is homeomorphic to ${}^\omega\omega$; 
so this is not a serious restriction.}
A \emph{tree forcing} is a preorder $\PP$, with respect to inclusion, where $\PP$ is a collection of perfect subtrees of ${}^{<\omega}\baseset$. 
A standard construction associates to each tree forcing $\PP$ an ideal
$\I_\PP$ on ${}^\omega \baseset$---we call ideals of this form
\emph{tree ideals}. Similarly, a family $\Meas({}^\omega
\baseset,\PP)$ of $\PP$-measurable subsets of ${}^\omega \baseset$ is
associated to each tree forcing. For details, see Sections~\ref{subsection what is a tree ideal}--\ref{section Positive Borel sets}.

Recall that $\sqsubseteq$ means \emph{initial segment} and the \emph{stem} of a tree $T$, denoted by $\stem(T)$, is the maximal node in $T$ that is $\sqsubseteq$-comparable with all other nodes in $T$. 

If $\I$ is a tree-ideal associated to $\PP$
define the set of \emph{$\I$-shift density points} (or short, just \emph{$\I$-density points}) of $A\subseteq {}^\omega\baseset$ as follows:
\begin{equation}\label{intro definition density points}
D_\I(A)=\{x\in{}^\omega \baseset\mid \exists m\in\omega\ \forall n\geq m\ \forall T \in \PP \; [\stem(T)=x{\upharpoonright}n \Rightarrow [T]\cap A \notin \I ] \}.
\end{equation}
In Section~\ref{section tree ideals} we prove the following theorem (as Corollary~\ref{density property for concrete forcings}); 
this section should be readable without prior knowledge of forcing.
\begin{theorem}\label{intro theorem D_I}
Let $\I$ be a strongly linked tree ideal (cf.\ Definition~\ref{definition strongly linked}). 
Then 
for any $A,B\in\Meas({}^\omega \baseset,\PP)$
\begin{enumerate-(A)}
\item\label{thm:D_I:well-def} $A =_\I B \Rightarrow D_\I(A)=D_\I(B)$, and 
\item\label{thm:D_I:positive} $D_\I(A) =_\I A$.
\end{enumerate-(A)}
\end{theorem}

We show in Section~\ref{section complexity} that $D_\I$ has reasonable complexity under the additional assumption that $\PP$ is Suslin. Then the map $D_\I{\upharpoonright}\Borel({}^\omega \baseset)$ 
is induced by an absolutely-$\Delta^1_2$ and therefore universally Baire measurable map $D \colon \Borelcodes \to \Borelcodes$ (where $\Borel({}^\omega \baseset)$ denotes the collection of Borel subsets of ${}^\omega \baseset$ and $\Borelcodes$ the set of their codes, as described above for $\RR$).

It follows immediately from Theorem~\ref{intro theorem D_I} that $D_\I$ is a useful notion of density point for the ideals associated to Cohen forcing, Hechler forcing, eventually different forcing, 
Laver forcing with a filter, and Mathias forcing with a shift invariant filter (cf.\ Section~\ref{section list of tree forcings}).
Although one can define $D_\I$ as in \eqref{intro definition density points} for arbitrary tree ideals, we verify in Section~\ref{section
  counterexamples} that the statement of Theorem~\ref{intro theorem D_I} fails for $\PP$ equal to Sacks,
Miller, Mathias, Laver, or Silver forcing.

While the null ideal is \emph{not} a strongly linked tree ideal (see Remark~\ref{lem:Random forcing no strongly linked dense subset}) our methods also yield a variant of density point for the null ideal.
Namely, when $\I=\I_\mu$ is the ideal of null sets with respect to the
Lebesgue measure $\mu$ on ${}^\omega2$ and $\PP$ is random forcing (see Definition \ref{def:listoftreeforcings}\ref{definition random})
let, for $A\subseteq {}^\omega 2$,
\begin{multline}\label{intro definition variant null ideal}
D_{\I_\mu}(A)=\Bigg\{x\in{}^\omega \baseset\mid \exists m\in\omega \;\forall n\geq m\; \forall T \in \PP 
\left[ \frac{\mu([T]\cap N_{x{\upharpoonright}n})}{\mu(N_{x{\upharpoonright}n})}> \frac{1}{2} \Rightarrow [T]\cap A \notin \I_\mu \right] \Bigg\}.
\end{multline}
We show in Section~\ref{sec: null ideal} that $D_{\I_\mu}(A)=_\mu A$ for any $\mu$-measurable set $A\subseteq {}^\omega 2$.

\medskip

One expects that for many ideals, no definable selector exists; in fact 
it turns out that this is the case
for the smallest non-trivial $\sigma$-ideal, the ideal of countable sets. 
We prove the following result in Theorem \ref{no selector for ideal of countable sets}:


\begin{theorem} \label{intro - no selector} 
Let $\I$ denote the ideal of countable subsets of ${}^\omega \baseset$.
There is no selector 
$\Borel({}^\omega \baseset)\to\Borel({}^\omega \baseset)$ for $=_\I$ 
which 
is induced by a universally Baire measurable function on Borel codes. 
\end{theorem} 

We extend this result in Lemma \ref{internal projective Cohen absoluteness} and Theorem \ref{no selector from internal absoluteness} 
to selectors with projective values, assuming the Axiom of Projective Determinacy ($\mathsf{PD}$). 
To state the next result, let $\Proj({}^\omega\baseset)$ denote the collection of projective subsets of ${}^\omega\baseset$ (we shall discuss the notion of codes for projective sets in Section~\ref{section: no selector}).

\begin{theorem} \label{intro - no projective selector} 
Assuming $\mathsf{PD}$, 
there is no selector 
$\Borel({}^\omega \baseset)\to\Proj({}^\omega \baseset)$ for $=_\I$ 
which is induced by a universally Baire measurable function on the codes. 
\end{theorem}

We further prove in Theorem~\ref{theorem models without selector} that it 
is consistent with $\ZFC$ that there is no selector which is definable by a first order formula in the language of set theory (with parameter a sequence of ordinals); and that it is consistent with $\ZF$ that there is no selector for $\I$ at all. In fact there is no such selector in Solovay's model. 

We would like to point out two interesting recent results that came to our attention after this paper was submitted. 
Firstly, there \emph{is} a definable selector on ${\bf \Delta}^0_2$ sets for $=_I$, where $I$ denotes the ideal of countable sets \cite{kanovei}. 
Secondly, the existence of a (not necessarily definable) \emph{lower density operator} (essentially, a selector preserving $\cap$) 
on Borel sets 
is equivalent to the Continuum Hypothesis \cite{balcerzakglab}.

\subsection*{What we aim for} \label{subsection - conclusion} 

Our aim is to find dividing lines between ideals
with and without a good notion of ``density point''.  The results show that 
for all tree forcings listed in Section \ref{section list of tree forcings}, 
the following three conditions (a)-(c) are
equivalent. Moreover for strongly linked tree forcings, random and Sacks forcing, all four conditions (a)-(d) are equivalent. 

\begin{enumerate-(a)} 
\item
$\PP$ is $\sigma$-linked. 
\item
$\PP$ satisfies the countable chain condition (ccc). 
\item
The analogue of Theorem \ref{intro theorem D_I} holds for $\I=\I_\PP$ and all $A,B\in \Borel({}^\omega\Omega)$, with 
$D_\I$ as in Definition~\ref{def:density points for ideals}. 
\item
There is a selector $\Borel({}^\omega\Omega)\rightarrow \Borel({}^\omega\Omega)$ for $=_{\I_{\PP}}$ that is induced by a universally Baire measurable function on the codes. 
\end{enumerate-(a)} 
Whether (d) holds for other well-known forcings is an important question left open in this paper (see Section~\ref{section open problems}). 

\medskip
\subsection*{Structure of the paper} We introduce density points for ideals in Section \ref{sec:dpideals} and study them for the null ideal in Section \ref{sec: null ideal}. 
In Sections \ref{subsection what is a tree ideal} to \ref{section Positive Borel sets}, we introduce tree ideals and study some of their properties. 
For instance, we show in Section \ref{section P-measurable sets} that for any tree forcing $\PP$ with the 
$\omega_1$-covering property, all Borel sets are 
$\PP$-measurable. This improves a result of Ikegami \cite[Lemma 3.5]{ikegami2010}. We then prove the main result on strongly linked tree ideals 
(Theorem \ref{intro theorem D_I}) 
in Sections
\ref{section equivalence to density property} and \ref{section
  strongly linked tree forcings}.  In Section \ref{section
  complexity}, we compute a bound on the complexity of density
operators for strongly linked tree ideals and show that they are universally Baire measurable. 
Section \ref{section counterexamples} contains counterexamples for the remaining tree
forcings listed in Section \ref{section list of tree forcings}.  In
Section \ref{section: no selector}, we prove that there are no selectors for the ideal of countable sets 
which are induced by universally Baire measurable functions (Theorems \ref{intro - no selector} and \ref{intro - no projective selector}). 
Moreover, we show
that it is consistent with $\ZF$ that there is no selector at all for
the ideal of countable sets. Section \ref{section list of tree
  forcings} contains a list of the tree forcings which we consider in this
article. 
As an additional result of independent interest, we show in Section \ref{section explicit construction of density points} that one can effectively construct density points of a closed set from a sequence of weights attached to basic open sets. 
We end with some open questions in Section \ref{section open
  problems}.

\subsection*{Acknowledgments}
The authors would like to thank Alessandro Andretta, Martin Goldstern, Daisuke Ikegami and Luca Motto Ros for discussions, 
Jindra Zapletal for the suggestion to study eventually different forcing and 
Vladimir Kanovei for reading a previous version of this article and sending us his version of the proof of Theorem~\ref{intro - no selector} (see \cite{kanovei-no-selector}). 
We further thank the anonymous referee for their work and many useful suggestions for improvement.

This project has received funding from the European Union's Horizon 2020 research and innovation programme under the Marie Sk\l odowska-Curie grant agreement No 794020 (IMIC) for the second-listed author and No 706219 (REGPROP) for the first, third, and last-listed authors. 
The first-listed author, formerly known as Sandra Uhlenbrock, was
partially supported by FWF grant number P28157, 
the second-listed author was partially supported by FWF grant number I4039, 
the third-listed author by the DNRF Niels Bohr Professorship
of Lars Hesselholt as well as FWF grant numbers F29999 and Y1012, and 
the last-listed author by FWF grants numbers I3081 and Y1012-N35. 
The first and second-listed authors further wish to thank the Erwin
Schr\"odinger International Institute for Mathematics and Physics
(ESI), Vienna for support during the workshop ``Current Trends in
Descriptive Set Theory'' in December 2016.

\section{Trees and density points}\label{sec:dpideals}

In this section we review some notation and introduce some terminology, including our notion of
density point, in a way that allows us to treat
\emph{both} 
$D_\I$, as defined in \eqref{intro definition density
  points} for strongly linked tree ideals, \emph{and} $D_{\I_\mu}$, our
variant of the density operator for the null ideal from \eqref{intro
  definition variant null ideal}, simultaneously. 







Recall that we write $\baseset$ to mean either $2$ (i.e., $\{0,1\}$) or $\omega$ (i.e., $\NN$). 
We consider subtrees $T$ of ${}^{<\omega}\baseset$ and write 
\[
[T] = \{x\in{}^\omega \baseset\mid(\forall n\in\omega)\; x{\upharpoonright}n \in T\}
\]
for the set of branches through $T$. 
A tree $T$ is \emph{perfect} if it has no end nodes and some splitting node above each node. 
Let 
$$s^\smallfrown T = \{ s^\smallfrown t \mid t \in T\}$$ 
$$T_u=\{t\in T\mid u\subseteq t \vee t\subseteq u\}\label{T_u}$$ 
for $s \in {}^{<\omega}\baseset$ and $u\in T$. 
For $s\in {}^{<\omega}\baseset$ and $U \subset {}^{\leq\omega}\baseset$ write
$$U/s  = \{t \in {}^{<\omega} \baseset \mid s^\smallfrown t\in U\}.$$
This is frequently used as $T/s$ when $T\subseteq {}^{<\omega}\baseset$ is a perfect tree or as $A/s$ when $A\subseteq {}^\omega\baseset$. In the latter case, $A/s$ is also called the \emph{detail} or \emph{localization of $A$ at $s$} and is often denoted by $A_{\lfloor s \rfloor}$ (see, e.g., \cite{andrettacamerlo2012}).

We write
$$\shift_s\colon {}^{\leq \omega}\baseset\rightarrow {}^{\leq \omega}\baseset, \ \shift_s(x)=s^\smallfrown x$$ 
denote the shift by $s\in {}^{<\omega}\baseset$. 
Thus $\shift_s^{-1}[T] = [T/s] =\{t\mid s^\smallfrown t\in [T]\}$. 
For $s,t \in \baseset^{\leq\omega}$, we write $s \sqsubseteq t$ if $s$
is an initial segment of $t$.
The longest $s\in T$ such
that $s\sqsubseteq t$ or $t \sqsubseteq s$ for all nodes $t \in T$ is called the \emph{stem} of $T$, denoted (as mentioned previously) by $\stem(T)$. 
The set of \emph{splitting nodes} of $T$ (those with at least two direct successors in $T$) is denoted $\spl_T$. 
Let $|t|$ denote the \emph{length} of a finite sequence $t\in {}^{<\omega}2$ (note that $|t|=\dom(t)$). 
Moreover, let $s {\wedge} t$ denote the longest common initial segment of $s$ and $t$. 

Of course an ideal $\I$ on a set $X$ is a collection $\I \subseteq \pow(X)$ such that $A,B \in \I$ implies $A\cup B \in \I$ and for any $C \subseteq B \in \I$, $C \in \I$. A $\sigma$-ideal is an ideal which is closed not just under finite, but under countable unions.
If $\I$ is an ideal on ${}^{\omega}\baseset$ and $A$ and $B$
are subsets of ${}^{\omega}\baseset$, recall that we write $A =_\I B$ for
$A \triangle B \in I$.
We also write $A \subseteq_\I B$ for $A\ \setminus\ B \in \I$
and $A\perp_\I B$ for $A\cap B\in \I$.  
An ideal on ${}^\omega\baseset$ is called \emph{shift invariant}\label{definition shift invariant ideal} if it is closed under pointwise images and preimages under $\shift_t$ for all $t\in{}^{<\omega}\baseset$. 
For an ideal $\I$, a set is called \emph{$\I$-positive} if it is not in $\I$; recall that the set of $\I$-positive sets (the co-ideal of $\I$) is denoted $\I^+$. 

\medskip
The central notion for this
article is the \emph{density property}:

\begin{definition}\label{definition density property}
If $\I$ is an ideal and $D$ is a map into $\pow({}^\omega \baseset)$ with $\Borel({}^\omega \baseset)\subseteq\dom(D)\subseteq \pow({}^\omega \baseset)$ and $A=_\I B\Rightarrow D(A)= D(B)$, we say that the \emph{density property} holds (with respect to $D$ and $\I$) if $D(A)=_\I A$ for all $A\in\dom(D)$. 
\end{definition}

Ideally, we would like to define a notion of density points relative to an arbitrary shift invariant ideal $\I$ on ${}^{\omega}\baseset$. 
For our definition we find it necessary to fix a collection $\bL_\I$ of sets which we consider ``large'';
but for strongly linked ideals and for the null ideal there is a natural choice of $\bL_\I$---namely when $\I=\I_\mu$, $\bL_{\I}$ is defined as the set of perfect sets of measure at least $\frac{1}{2}$ and
when $\I=\I_\PP$ and $\PP$ is strongly linked, $\bL_{\I}$ is defined to be the set
$\{[T]\in\PP\mid \stem(T)=\emptyset\}$.

Since we want to speak about arbitrary tree ideals in some of our
results below, we make the following convention.
\begin{convention}\label{convention}
Let $\I$ be a tree ideal, and fix $\PP$ such that $\I=\I_\PP$.
If $\I =\I_\mu$ we shall assume that $\PP$ is random forcing, i.e. the collection of trees $T\subseteq {}^{<\omega}2$ such that for all $s \in T$, $\mu([T]\cap N_s)>0$; further, we let
\[
\bL_\I = \{[T]  \mid T \in \PP, \mu([T])>\frac{1}{2}\}.
\]
If $\PP$ is any other tree forcing then we let
\[
\bL_\I = \left\{[T]  \mid T \in \PP,  \stem(T) = \emptyset\right\}.
\]
We say that elements of $\bL_\I$ are \emph{large} with respect to $\I$. 
\end{convention}

\begin{definition} \label{def:density points for ideals}
  Suppose that $A$ is a subset of ${}^{\omega}\baseset$. 
\begin{enumerate-(a)} 
\item \label{def:density points 1}
An element $x$ of ${}^{\omega}\baseset$ is an \emph{$\I$-shift density
   point of $A$} if there is some $n_x$ such that
  for all $B \in \bL_\I$ and $n\geq n_x$ 
  \[ \shift_{x\upharpoonright n}(B)\cap A \notin \I. \]
\item 
  $D_\I(A)$ 
  denotes the set of $\I$-shift density points of $A$. 
\end{enumerate-(a)} 
We further say that the \emph{$\I$-shift density property} holds if $D_\I(A)\triangle A\in \I$ for all Borel sets $A$. 
\end{definition}

For simplicity, we sometimes just write 
$\I$-density point and
$\I$-density property. 
It is clear that by Convention~\ref{convention}, Definition~\ref{def:density points for ideals} just repeats the defintion of $D_\I$ given in 
\eqref{intro definition density points} for strongly linked tree ideals as well as the one in \eqref{intro definition variant null ideal}
for the null ideal.
Note that 
$D_\I(A)$ is $\mathbf{\Sigma}^0_2$ for any
subset $A$ of ${}^\omega\baseset$. 
To see this, let 
\[
S = \{s \in {}^{<\omega} \baseset \mid \forall B\in  \bL_\I\ \sigma_s(B)\cap A\notin\I\}
\] 
and observe that $x \in D_\I(A) \iff \exists m\ \forall n\geq m\ x{\upharpoonright}n \in S$; thus $D_\I(A)$ is $\Sigma^0_2(S)$. 

Finally, note that in those cases where we verify the $\I$-shift density property, we obtain that $D_\I(A)\triangle A\in \I$ for all $\PP$-measurable\footnote{$\PP$-measurability is defined in Section \ref{section P-measurable sets}. } 
subsets $A$ of ${}^\omega \baseset$ rather than 
just for Borel sets, since in all these cases $\PP$ is ccc (cf.\ Remark~\ref{remark ccc}). 

\begin{remark} 
Definition \ref{def:density points for ideals} can be rephrased in the following fashion. 
We call a subset $A$ of ${}^\omega\baseset$ \emph{$\I$-full} if
for all $B \in \bL_\I$, the set $B \cap A$ is $\I$-positive (this is analogous to the definition of stationary sets from club sets). 
Then $x$ is a density point of $A$ if and only if $A/(x{\upharpoonright}n)$---which by definition is the same as $\sigma_{x{\upharpoonright}n}^{-1}(A)$---is eventually $\I$-full as $n$ increases.\footnote{As mentioned above, the localization $A/s$ of $A$ at $s$ is often denoted by $A_{\lfloor s \rfloor}$ (see e.g. \cite{andrettacamerlo2012}).}
\end{remark}

\begin{remark}\label{remark poreda} 
We notice that Definition \ref{def:density points for ideals} can also be rephrased via the following notion of convergence. 
    We say that a sequence  $\vec{f}=\langle f_n\mid n\in\omega\rangle$ of functions $f_n\colon {}^\omega\baseset \rightarrow \RR$ \emph{converges in $\I$} to a function $f\colon {}^\omega\baseset \rightarrow \RR$ if the following condition holds:\footnote{Compare this with convergence in measure as in Lemma \ref{convergence in measure}.}  
    For all $\epsilon>0$, there is some $n_0$ such that for all $B\in \bL_\I$ and $n\geq n_0$, 
    $$B\ \setminus\ \{x\in {}^\omega\baseset \mid |f_n(x)-f(x)|\geq \epsilon \}\notin \I.$$ 
    By shift invariance, the condition $\shift_{x\upharpoonright n}(B)\cap A \notin \I$ in Definition \ref{def:density points for ideals} \ref{def:density points 1} is equivalent to $B\cap \shift_{x\upharpoonright n}^{-1}(A) \notin \I$. Moreover, $B\cap \shift_{x\upharpoonright n}^{-1}(A)=B\ \setminus\ \{x\in {}^\omega\baseset\mid |1_{\shift_{x{\upharpoonright}n}^{-1}(A)}(x)-1|\geq\epsilon\}$ for any $\epsilon$ with $0<\epsilon<1$. 
    Therefore $x$ is an $\I$-shift density point of $A$ if and only if the sequence 
    $\langle 1_{\shift_{x{\upharpoonright}n}^{-1}(A)} \mid n\in\omega\rangle$ 
    of characteristic functions of $\shift_{x{\upharpoonright}n}^{-1}(A)$ 
    converges in $\I$ to the constant function with value $1$. 
\end{remark}

\section{The null ideal} 
\label{sec: null ideal}

We now outline the situation in the special case of the $\sigma$-ideal of Lebesgue null subsets of ${}^\omega2$ 
to illustrate some ideas used in this paper.


Recall that $\mu$ denotes Lebesgue measure on ${}^\omega2$ and $\I_\mu$ the $\sigma$-ideal of 
$\mu$-null
sets. 
The next lemma 
implies that the $\I_\mu$-shift density property holds. 

\begin{lemma} \label{lem:separating Lebesgue density} 
Let $x$ be an element and $A$ a Borel subset of ${}^{\omega}2$. 
\begin{enumerate-(1)} 
\item 
If $d^\mu_A(x)=1$ and $\epsilon>0$, then there is some $n_{x,\epsilon}$ such that for all $n\geq n_{x,\epsilon}$ and all Borel sets $B$ with $\mu(B)\geq\epsilon$, $\mu(\sigma_{x{\upharpoonright}n}(B)\cap A)>0$. 
\item 
If $d^\mu_A(x)=0$ and $\epsilon>0$, then there is a 
Borel set $B$ with 
$$
\exists^\infty n\ (\shift_{x\upharpoonright n}
(B)\cap A=\emptyset)$$ 
and $\mu(B)\geq 1-\epsilon$. 
\end{enumerate-(1)}
\end{lemma} 
\begin{proof} 
  For the first claim, note that there is some $n_{x,\epsilon}$ with $\frac{\mu(A\cap N_{x\upharpoonright n})}{\mu(N_{x{\upharpoonright}n})}> 1-\epsilon$ for all $n\geq n_{x,\epsilon}$, since 
  \[
  d^\mu_A(x)=\liminf_n \frac{\mu(A\cap N_{x{\upharpoonright}n})}{\mu(N_{x{\upharpoonright}n})}=1. 
  \]
  If $B$ is any $\I_\mu$-positive Borel set of size at least $\epsilon$, 
  then $\mu(\sigma_{x{\upharpoonright}n}(B)\cap A)>0$ for all $n\geq n_{x,\epsilon}$.  

  For the second claim, let 
  $\vec{\epsilon}=\langle \epsilon_i\mid i<\omega\rangle $ be a
  sequence in $\RR^+$ with $\sum_i \epsilon_i\leq\epsilon$. 
  Since 
  \[
  d^\mu_A(x)=\liminf_n \frac{\mu(A\cap N_{x{\upharpoonright}n})}{\mu(N_{x{\upharpoonright}n})}=0, 
  \]
  there is a strictly
  increasing sequence $\vec{n}=\langle n_i\mid i\in\omega\rangle$ with
  $\frac{\mu(A\cap N_{x\upharpoonright n_i})}{\mu(N_{x\upharpoonright
      n_i})}<\epsilon_i$ for all $i\in\omega$. 
  Let $B_i=\shift^{-1}_{x\upharpoonright n_i}(A)$ 
  for $i\in\omega$ 
  and $B=\bigcup_{i\in\omega} B_i$.  Then
  $$\mu(B) \leq \sum_i \mu(B_i) \leq \sum_i \frac{\mu(A\cap N_{x\upharpoonright n_i})}{\mu(N_{x\upharpoonright n_i})} \leq\sum_i \epsilon_i\leq\epsilon .$$
  Let $C$ be an $\I_\mu$-positive set with $\mu(C)\geq 1-\epsilon$ that is disjoint from $B$. Then $C\cap B_i=\emptyset$ for all $i\in\omega$.  
  Since
  $B_i=\shift^{-1}_{x\upharpoonright n_i}(A)$, it follows that
  $\shift_{x\upharpoonright n_i} (C)\cap A=
  \shift_{x\upharpoonright n_i} (C) \cap
  \shift_{x\upharpoonright n_i}(B_i)=\shift_{x\upharpoonright n_i}(C\cap B_i)=\emptyset$.
\end{proof} 

For $\epsilon=\frac{1}{2}$, we obtain that $d^\mu_A(x)=1$ implies that $x$ is an $\I_\mu$-shift density point and $d^\mu_A(x)=0$ implies that this fails. 
Using Lebesgue's density theorem for $\mu$ and ${}^\omega 2$ (see \cite[Section 8]{andrettacamerlo2012} or \cite[Proposition 2.10]{Mi08}) this yields the $\I_\mu$-shift density property  (cf.\ Definition~\ref{def:density points for ideals}). 

\begin{corollary}
  For every Lebesgue measurable subset $A$ of ${}^\omega2$,
  $D_{\I_\mu}(A)=_\mu D_\mu(A)$, recalling that by $D_\mu(A)$ we denote the points of $A$ of density 1 (see p.~\pageref{density1}).  In particular, the $\I_\mu$-shift density
  property holds.
\end{corollary}
We shall give another proof of Lebesgue's density theorem in Section~\ref{section explicit construction of density points} below, thus making the previous argument for the $\I_\mu$-shift density property self-contained. 
The $\I_\mu$-density property also follows as a special case from the results in Section \ref{section strongly linked tree forcings}. 

\medskip
In the next lemma, we give two examples which show that if $d^\mu_A(x) \in (0,1)$, then $x$ can but does not have to be an $\I_\mu$-shift density point of $A$. 

\begin{lemma} \label{density versus shift density} 
  Each of the following statements is satisfied by some Borel subset $A$ of ${}^\omega2$ and some $x\in {}^{\omega}2$  with $d^\mu_A(x) \in (0,1)$. 
\begin{enumerate-(a)} 
\item \label{density versus shift density 1} 
  $x$ is an $\I_\mu$-shift density point of $A$. 
\item \label{density versus shift density 2} 
  $x$ is not an $\I_\mu$-shift density point of $A$. 
\end{enumerate-(a)}
\end{lemma}  
\begin{proof} 
  Let $A=\{0^n{}^\smallfrown 1^3{}^\smallfrown x\in {}^\omega2\mid n\in\omega,\ x\in {}^\omega 2\}$ and $B$ its complement. 
    
  For \ref{density versus shift density 1} note that $d^\mu_B(0^\omega)\in (0,1)$ since $\frac{\mu(B\cap N_{0^n})}{\mu(N_{0^n})}=\frac{3}{4}$ for all $n\in\omega$.  
  Thus $\mu(\sigma_{0^n}(C)\cap B)>0$ for any Borel set $C$ with $\mu(C)\geq\frac{1}{2}$.  
  So $0^\omega$ is an $\I_\mu$-shift density point of $B$. 
  
  For \ref{density versus shift density 2} we have $d^\mu_A(0^\omega)=1-d^\mu_B(0^\omega)\in (0,1)$. Since $\mu(B)\geq\frac{1}{2}$ and $\sigma_{0^n}(B)\cap A=\emptyset$ for all $n\in\omega$, $0^\omega$ is not an $\I_\mu$-shift density point of $A$. 
  \end{proof} 


\np


\section{Tree ideals}\label{section tree ideals} 

In this section, we study ideals induced by collections of trees. 
We introduce the class of strongly linked tree ideals 
and show that the shift density property holds for this class. 
Recall again that we work in the Polish space ${}^\omega\baseset$ where $\baseset$ is either $2$ (i.e., $\{0,1\}$) or $\omega$ (i.e., $\NN$).


\subsection{What is a tree ideal?} \label{subsection what is a tree ideal}

A tree ideal on ${}^\omega \baseset$ is induced by
a collection $\PP$ of perfect subtrees of ${}^{<\omega}\baseset$ that contains ${}^{<\omega}\baseset$ and $T_s$ for all $T\in\PP$ and $s\in T$. 
We will always assume this condition for any collection of trees.

Any such collection of trees carries the partial order $S\leq T:\Longleftrightarrow S\subseteq T\Longleftrightarrow [S]\subseteq [T]$. 
Such partial orders are also called \emph{tree forcings};\footnote{These forcings, but without the condition that ${}^\omega\baseset$ is in $\PP$, are called \emph{strongly arboreal} in \cite[Definition 2.4]{ikegami2010}. } 
some well-known examples are listed in Section \ref{section list of tree forcings}. 
For instance, the null ideal is induced by the collection of random trees, given as follows: 

\begin{example}\label{example random} 
A subtree $T$ of ${}^{<\omega}2$ is \emph{random} if $\mu([T_s])>0$ for all $s\in T$ with $\stem(T) \sqsubseteq s$. 
\end{example} 



We next associate an ideal to any collection of trees (we follow \cite[Definition 2.6]{ikegami2010}). 
The underlying idea is based on the special case that 
the sets $[T]$ for $T\in \PP$ form a base for a topology. In this case, $\PP$ is called \emph{topological} and its topology is denoted $\tau_\PP$. 
For instance, the collection of Hechler trees (see Definition
\ref{def:listoftreeforcings} \ref{definition Hechler}) is
topological. 
In fact, all strongly linked collections of trees
(as defined in Section \ref{section strongly linked tree forcings}) have this
property.\footnote{Topological does not imply the density
  property. For instance, assuming $\CH$ one can construct a dense
  topological (shift-invariant) subforcing of Sacks forcing, while we
  show in Propostion \ref{counterexamples for Mathias, Sacks, Silver}
  and Theorem \ref{no selector for ideal of countable sets} that the
  density property fails for the ideal associated to Sacks forcing. } 

Usually, one defines nowhere dense sets relative to a given topology, or equivalently, to a base of that topology. 
Moreover, meager sets are defined as countable unions of these sets. 
In the next definition, these notions are generalized by replacing a base by an arbitrary collection of trees. 

\begin{definition} \cite[Definition 2.6]{ikegami2010} 
Let $\mathbb{P}$ be a collection of trees. 
\begin{enumerate-(a)} 
\item 
A set $A$ is \emph{$\PP$-nowhere dense} if for all $T\in \mathbb{P}$ there is some $S \leq T$ with $[S]\cap A=\emptyset$. Moreover, $\cN_{\mathbb{P}}$ is the ideal of $\PP$-nowhere dense sets. 
\item 
$\cI_\PP$ is the $\sigma$-ideal of \emph{$\PP$-meager sets} generated by $\cN_\PP$.
\end{enumerate-(a)} 
\end{definition} 

\emph{Tree ideals} are those of the form $\cI_\PP$ for a collection $\PP$ of trees.\footnote{In \cite[Section 2]{brendleloewe} and various other papers, tree ideals mean the ideals $\cN_\PP$ instead of $\cI_\PP$. }
This presentation allows for uniform proofs of results for various ideals. 
Moreover, many well-known ideals are of this form; 
for instance, for Cohen forcing\footnote{See Section \ref{section list of tree forcings} for this and the following forcings. } $\tau_\PP$ is the standard topology, so $\cN_\PP$ is the collection of nowhere dense sets and $\cI_\PP$ that of meager sets. 
For random forcing, $\cN_\PP$ and $\cI_\PP$ equal the $\sigma$-ideal of null sets. 
Sacks forcing is the collection of all perfect trees; here both ideals equal the \emph{Marczewski ideal} (see \cite[3.1]{szpilrajn1935classe} or for a more modern presentation, \cite[p.~306]{bukovsky}). Its restriction to Borel sets equals the ideal of countable sets by the perfect set property for Borel sets. For Mathias forcing, $\tau_\PP$ is the Ellentuck topology, and $\cN_\PP$ and $\cI_\PP$ are equal to the ideal of nowhere Ramsey sets (see \cite{brendlekhomskiiwohofsky}). 
The ideal associated to Silver forcing consists of the completely doughnut null sets (see \cite{halbeisen2003making}). 


\np


\subsection{Measurability for tree ideals} \label{section P-measurable sets} \label{section Borel sets are P-measurable}  

Let $\PP$ be a collection of subtrees of ${}^{<\omega}\baseset$. 
The next definition introduces a form of indivisibility\footnote{See e.g. \cite{laflamme2009partitions} for the concept of indivisibility in combinatorics. }  of $\PP$ with respect to $A$: If $T\in\PP$ and $[T]$ is split into the two pieces $[T]\cap A$ and $[T]\setminus A$, then at least one of these pieces contains a set of the form $[S]$ for some $S\in\PP$, up to some $\PP$-meager set. 

\begin{definition} \cite[Definition 2.8]{ikegami2010}\footnote{This is a variant of a definiton in \cite[Section 0]{MR2127234}. } \label{def:Pmeasurable}
A subset $A$ of ${}^\omega\baseset$ is called \emph{$\PP$-measurable} if for every $T\in \PP$, there is some $S\leq T$ with at least one of the properties (a) $[S]\subseteq_{\cI_\PP} A$ and (b) $[S]\perp_{\cI_\PP}A$. 
The collection of $\PP$-measurable sets is denoted $\Meas({}^\omega\baseset,\PP)$. 
\end{definition} 

Note that the properties (a) and (b) are mutually exclusive (see Lemma \ref{characterization of positive sets} below). 

The main motivation for introducing this notion is the fact that it formalizes various well-known properties. For instance, we will see in Lemma \ref{measurability for the null ideal} that $\PP$-measurability for random forcing means Lebesgue measurability. 
For Sacks forcing it is equivalent to the Bernstein property for collections of sets closed under continuous preimages and intersections with closed sets \cite[Lemma 2.1]{brendleloewe}. 
Moreover, for Mathias forcing it is equivalent to being completely Ramsey. 

\np



Our next goal is to show that for a very large class of forcings, all Borel sets are $\PP$-measurable. 
This will be important in the proofs of the following sections. 

\cite[Lemma 3.5]{ikegami2010} shows that for proper tree forcings $\PP$, all Borel sets are $\PP$-measurable.\footnote{The proof of this and several other results in \cite{ikegami2010} can also be done from the weaker assumption that $\PP$ has the $\omega_1$-covering property. 
We give a more direct proof. }  
We will show a slightly more general version of this result. 
To state this, 
recall that a forcing $\PP$ has the \emph{$\omega_1$-covering property} if for any $\PP$-generic filter $G$ over $V$, any countable set of ordinals in $V[G]$ is covered by (i.e., is a subset of) a set in $V$ that is countable in $V$. 
For instance, this statement holds for all proper and thus for all  Axiom A forcings.\footnote{See \cite[Definition 31.10]{MR1940513}. } In particular, it holds for all forcings considered in this paper. 

We will need the following characterization of the $\omega_1$-covering property. 
To state it, we introduce the following notation: For 
$D\subseteq \PP$ and $p\in \PP$ let us write
\[
D^{\parallel p}=\{q\in D\mid p\parallel q\},
\] 
where $p\parallel q$ denotes that
there is an $r \in \PP$ such that $r \leq p$ and $r \leq q$. 

\begin{lemma} \label{preservation of omega1} 
The following conditions are equivalent for any forcing $\PP$: 
\begin{enumerate-(a)} 
\item \label{preservation of omega1 a} 
$\PP$ has the $\omega_1$-covering property. 
\item \label{preservation of omega1 b} 
For any condition $p\in\PP$ and any sequence $\vec{D}=\langle D_n\mid n\in\omega\rangle$ of antichains in $\PP$, there is some $q\leq p$ such that for any $n\in\omega$, the set 
$D_n^{\parallel q}$ 
is countable. 
\end{enumerate-(a)} 
\end{lemma} 
\begin{proof} 
  We first assume \ref{preservation of omega1 a}. Let $p\in\PP$ and
  let $\vec{D}=\langle D_n\mid n\in\omega\rangle$ be as in
  \ref{preservation of omega1 b}.  Take a $\PP$-generic filter $G$
  over $V$.  Moreover, let $f(n)$ be an element of $D_n\cap G$
  for each $n\in\omega$.  By the $\omega_1$-covering property, there
  is a countable subset $C\in V$ of $\PP$ such that $f(n)\in C$ for
  all $n\in\omega$.  Let $\dot{f}$ be a name for $f$ such that
  $q\leq p$ forces $\dot{f}(n)\in\check{C}$ for all $n\in\omega$.  It
  follows that $D_n^{\parallel q}\subseteq C$,
since for any $\PP$-generic filter $H$ over $V$ that contains both $q$ and $r$ we have $r=\dot{f}^H(n)\in C$. 

For the converse implication, assume \ref{preservation of omega1 b} and suppose that $G$ is $\PP$-generic over $V$ and $C$ is a countable set of ordinals in $V[G]$. Moreover, let $f$ be an enumeration of $C$ and $\dot{f}$ a name with $\dot{f}^G=f$. 
Then there is a condition $p\in G$ which forces that $\dot{f}\colon \omega\rightarrow\Ord$ is a function. 
For each $n\in\omega$, let $D_n$ be a maximal antichain of conditions deciding $\dot{f}(n)$. 
By our assumption, there are densely many conditions $q\leq p$ as in \ref{preservation of omega1 b}. Hence there is some $q\in G$ as in \ref{preservation of omega1 b}. 
Since $D_n^{\parallel q}$ 
is countable for all $n\in\omega$,
$C_n = \{ \alpha \mid r \Vdash \dot f(n) = \alpha \text{ for some } r
\leq q,q^\prime \text{ for a } q^\prime \in D_n^{\parallel q} \}$ is
countable and hence $C=\bigcup_{n\in\omega} C_n$ is a countable cover
of $\ran(f)$.
\end{proof} 


To show that all Borel sets are $\PP$-measurable if $\PP$ has the $\omega_1$-covering property, we need the next two easy lemmas. 
We will use the following notation. 
If $A$ is a subset of ${}^\omega\baseset$, let $\PP_A=\{T\in \PP\mid [T]\subseteq_{\cI_\PP}A\}$. 
We further say that a subset of $\PP$ is \emph{$A$-good} if it is contained in $\PP_{(A)}=\PP_{A}\cup \PP_{{}^\omega\baseset\setminus A}$. 

\begin{lemma} \label{characterization of measurability} 
A subset $A$ of ${}^\omega\baseset$ is $\PP$-measurable if and only if
there is an $A$-good maximal antichain in $\PP$. 
\end{lemma} 
\begin{proof} 
If $A$ is $\PP$-measurable, then $\PP_{(A)}$ is a dense subset of
$\PP$. Then there is a maximal antichain in $\PP$ contained in
$\PP_{(A)}$ and hence an $A$-good maximal antichain. 
Conversely, if $D$ is a maximal $A$-good antichain in $\PP$ and $S\in\PP$, let $T\in D^{\parallel S}$ and $R\leq S,T$. Then $[R]\subseteq_{\cI_\PP} A$ or $[R]\subseteq_{\cI_\PP} {}^\omega\baseset \setminus A$. 
\end{proof} 

If $D\subseteq \PP$, write 
\[
\sideset{}{^\square} \bigcup D=\bigcup_{T\in D}[T].
\] 

\begin{lemma} \label{complement of antichain is small} 
If $D$ is a maximal antichain in $\PP$ and $T\in\PP$, then $[T] \setminus \bigcup^\square D^{\parallel T} \in \cN_\PP$. 
\end{lemma} 
\begin{proof} 
Let $S\in\PP$, $R\in D^{\parallel S}$ and $Q\leq R,S$. 
If $Q\perp T$, then there is some $P\leq Q$ with $[P]\cap[T]=\emptyset$ by the closure property of $\PP$ 
defined in the beginning of Section \ref{subsection what is a tree ideal}. 
If $Q\parallel T$, let $P\leq Q,T$. 
Since $P\in D^{\parallel T}$, $[P]$ is disjoint from $[T] \setminus \bigcup^\square D^{\parallel T}$, as required. 
\end{proof}

The next result shows that Borel sets are $\PP$-measurable in all relevant cases. 

\begin{lemma} \label{P-measurable sets form a sigma-algebra} 
If $\PP$ has the $\omega_1$-covering property, then the $\PP$-measurable sets form a $\sigma$-algebra. 
In particular, all Borel sets are $\PP$-measurable. 
\end{lemma} 
\begin{proof} 
For the first claim, it suffices to show that the class of $\PP$-measurable sets is closed under forming countable unions. 
To this end, let $\vec{A}=\langle A_n\mid n\in\omega\rangle$ be a sequence of $\PP$-measurable subsets of ${}^\omega\baseset$. 
Furthermore, let $D_n$ be an $A_n$-good maximal antichain for each $n\in\omega$ by Lemma \ref{characterization of measurability}. 
We will show that $A=\bigcup_{n\in\omega} A_n$ is $\PP$-measurable. 

Fix any $T\in \PP$. 
Since $\PP$ has the $\omega_1$-covering property, there is some $S\leq T$ such that the sets $E_n=D_n^{\parallel S}$ 
are countable for all $n\in\omega$ by Lemma \ref{preservation of omega1}. 
First assume that for some $n\in\omega$, there is a tree $R\in E_n$ with $[R]\subseteq_{\cI_\PP} A_n$. 
Then there is some $Q\leq S$ with $[Q]\subseteq_{\cI_\PP} A_n\subseteq A$ as required. 
So we can assume that the previous assumption fails. 
We claim that then 
$[S]\perp_{\cI_\PP} A$. 
It suffices to show that $[S]\cap A_n \in \cI_\PP$ for each $n\in\omega$, since $\cI_\PP$ is a $\sigma$-ideal. 
To see this, note that $[R]\cap A_n\in \cI_\PP$ for all $R\in E_n$ by our case assumption. Hence $\bigcup^\square E_n\cap A_n\in \cI_\PP$. 
Moreover, $[S]\setminus \bigcup^\square E_n\in\cN_\PP$ 
by Lemma \ref{complement of antichain is small} and therefore $[S]
\cap A_n \in \cI_\PP$. 

Since it easy to see that closed sets are $\PP$-measurable, it follows that all Borel sets are $\PP$-measurable. 
\end{proof} 


Next is the observation that $\PP$-measurability for random forcing means Lebesgue measurability. 
We recall the argument from \cite{ikegami2010} for the benefit of the reader. 

\begin{lemma} \label{measurability for the null ideal} \cite[Proposition 2.9]{ikegami2010}
If $\PP$ is a ccc tree forcing and $A$ is any subset of ${}^\omega\baseset$, then the following conditions are equivalent. 
\begin{enumerate-(a)} 
\item 
$A$ is $\PP$-measurable. 
\item 
There is a Borel set $B$ with $A\triangle B\in \cI_\PP$. 
\end{enumerate-(a)} 
\end{lemma} 
\begin{proof} 
If $A$ is $\PP$-measurable, then $\PP_{(A)}$ is dense in $\PP$. 
Let $D\subseteq \PP_{(A)}$ be a maximal antichain in $\PP$. 
Since $D$ is countable, the sets $B_0=\bigcup^\square (D\cap \PP_A)\subseteq_{\cI_\PP}A$ and $B_1=\bigcup^\square (D\setminus \PP_A)\subseteq_{\cI_\PP}{}^\omega\baseset\setminus A$ are Borel. 
Since $D$ is maximal, ${}^\omega\baseset\setminus \bigcup^\square D\in\cN_\PP$ by Lemma \ref{complement of antichain is small}. 
Thus $A\triangle B_0\in\cI_\PP$. 

The reverse implication follows from the fact that all Borel sets are $\PP$-measurable by Lemma \ref{P-measurable sets form a sigma-algebra}. 
\end{proof}

\np


\subsection{Positive Borel sets} \label{section Positive Borel sets} 

The following characterization of positive Borel sets via trees will be important below. It uses an auxiliary ideal from \cite{ikegami2010}. 

\begin{definition} \cite[Definitions 2.11]{ikegami2010} 
$A\in \cIs_\PP$ 
if for all $T\in \mathbb{P}$, there is some $S \leq T$ with $[S]\cap A\in \cI_{\mathbb{P}}$.
\end{definition} 


It is clear that $\cI_\PP\subseteq \cIs_\PP$, but it is open whether equality holds for all proper tree forcings.\footnote{To our knowledge, every known proper tree forcing satisfies either the ccc or fusion and thus equality holds.} 
Note that equality holds for ccc forcings. 
To see this, assume that $A\in \cIs_\PP$. Then $\PP_{{}^\omega\baseset\setminus A}$ (as defined before Lemma \ref{characterization of measurability}) is dense in $\PP$ and therefore contains a (countable) maximal antichain $D$. 
We have $A\cap \bigcup^\square D \in \cI_\PP$.\footnote{This notation is defined before Lemma \ref{complement of antichain is small}. } 
Since ${}^\omega\baseset\setminus \bigcup^\square D\in \cN_\PP$ by Lemma \ref{complement of antichain is small}, we have $A\in \cI_\PP$. 
Moreover, equality holds for fusion forcings (then in fact $\cN_\PP=\cI_\PP$)\footnote{For a tree forcing $\PP$, we define \emph{fusion} as the existence of a sequence $\vec{\leq}=\langle \leq_n\mid n\in\omega \rangle$ of partial orders on $\PP$ with $\leq_0=\leq$ which satisfy the following conditions: 
  \begin{enumerate-(a)} 
  \item (\emph{decreasing})
  If $S\leq_n T$ and $m\leq n$, then $S\leq_m T$. 
  \item (\emph{limit})
  If $\vec{T}=\langle T_n\mid n\in\omega\rangle$ is a sequence in $\PP$ with $T_{n+1}\leq_n T_n$ for all $n\in\omega$, then there is some $S\in \PP$ with $S\leq_n T_n$ for all $n\in\omega$. 
  \item (\emph{covering})
  If $T\in \PP$, $n\in\omega$ and $D$ is dense below $T$, then there is some $S\leq_n T$ with $[S]\subseteq\bigcup^\square D$. 
  \end{enumerate-(a)} 
} and for topological tree forcings as in the proof of \cite[Lemma 3.8]{FKK16}. 

The next lemma characterizes $\cIs_\PP$-positive sets. 

\begin{lemma} \label{characterization of positive sets} 
Suppose that $\PP$ is a tree forcing.
\begin{enumerate-(1)} 
\item 
For any $T\in \PP$, $[T]\notin \cIs_\PP$. 
\item 
A $\PP$-measurable subset $A$ of ${}^\omega\baseset$ is $\cIs_\PP$-positive if and only if there is some $T\in\PP$ with $[T]\subseteq_{\cIs_\PP} A$. 
\end{enumerate-(1)} 
\end{lemma} 
\begin{proof} 
We show the first claim. 
If $[T]\in\cIs_\PP$, then $[S]\in \cI_\PP$ for some $S\leq T$. Let
$\vec{A}=\langle A_n\mid n\in\omega\rangle$ be a sequence of sets in $\cN_\PP$ with $[S]\subseteq \bigcup_{n\in\omega} A_n$. 
We can then recursively construct a sequence $\vec{S}=\langle S_n\mid n\in\omega\rangle$ in $\PP$ such that $S_0=S$, $S_{n+1}\subseteq S_n$, $[S_n]\cap A_n=\emptyset$ for all $n\in \omega$ and the sequence $\vec{s}=\langle \stem(S_n)\mid n\in\omega\rangle$ of stems is strictly increasing. Then $x=\bigcup_{n\in\omega} \stem(S_n)\in \bigcap_{n\in\omega} [S_n]$. Since  $[S_n]\cap A_n=\emptyset$ for all $n\in\omega$ and $[S]\subseteq \bigcup_{n\in\omega} A_n$, we have $x \notin [S]$. But this contradicts the fact that $x\in [S]$. 

We now show the second claim. 
By the first part, it is sufficient to take any $\PP$-measurable $\cIs_\PP$-positive set $A$ and find some $S\in \PP$ with $[S]\subseteq_{\cIs_\PP} A$. 
Assume that there is no such tree. 
Since $A$ is $\PP$-measurable, we have that for any $T\in \PP$ there
is some $S\leq T$ with $[S]\perp_{\cI_\PP}A$. Hence $A\in \cIs_\PP$ by
the definition of $\cIs_\PP$. 
\end{proof} 

For ideals $\I$ of the form $\cIs_\PP$ such that $\PP$ has the $\omega_1$-covering property, the definition of $\I$-shift density points (Definition \ref{def:density points for ideals}) of a Borel set $A$ can now be formulated in the following way: 
An element $x$ of ${}^{\omega}\baseset$ is an \emph{$\I$-shift density point of $A$} if there is some $n_x$ such that for all $B \in \bL_\I$ and $n\geq n_x$, there is some $T\in \PP$ with 
  \[[T]\subseteq_I \shift_{x\upharpoonright n}(B)\cap A. \] 

 Note that it is easy to see that $\PP$-measurability remains
equivalent if $\cI_\PP$ is replaced with $\cIs_\PP$. 
Moreover, the ideals $\cN_\PP$, $\cI_\PP$ and $\cIs_\PP$ remain the same if $\PP$ is replaced by a dense subset by the next remark.

\begin{remark} 
$\PP$ is dense in $\QQ$ if for every $T\in\PP$, there is some $S\leq T$ in $\QQ$. 
We define $\PP$ and $\QQ$ to be \emph{mutually dense} if $\PP$ is dense in $\QQ$ and conversely. 

We claim that $\cN_\PP=\cN_\QQ$, $\cI_\PP=\cI_\QQ$ and $\cIs_\PP=\cIs_\QQ$ if $\PP$ and $\QQ$ are mutually dense. 
To see that $\cN_\PP\subseteq \cN_\QQ$, take any $A\in \cN_\PP$ and 
$T\in \QQ$. As $\PP$ is dense in $\QQ$, there is some $T^\prime \leq T$ in $\PP$. Since $A\in \cN_\PP$, there is some $S\leq T^\prime$ in $\PP$ with $[S]\cap A =\emptyset$. As $\QQ$ is dense in $\PP$, there is some $S^\prime \leq S$ with $S^\prime \in \QQ$ such that $[S^\prime] \cap A \subseteq [S]\cap A =\emptyset$. 
Thus $\cN_\PP= \cN_\QQ$ and $\cI_\PP=\cI_\QQ$. 
A similar argument shows $\cIs_\PP=\cIs_\QQ$. 

Conversely, any two collections of trees $\PP$ and $\QQ$ with the $\omega_1$-covering property and $\cIs_\PP=\cIs_\QQ$ are mutually dense by Lemmas \ref{P-measurable sets form a sigma-algebra}  and \ref{characterization of positive sets}. 
\end{remark}


\np


\subsection{An equivalence to the density property} \label{section equivalence to density property} 

Let $\I$ be an ideal on ${}^\omega\baseset$. 
We say that a function $D\colon \Borel({}^\omega \baseset)\rightarrow \Borel({}^\omega \baseset)$ is \emph{lifted from $\Borel({}^\omega \baseset)/\I$} if for any $A, B \in \Borel({}^\omega \baseset)$,  $A=_\cI B \Longrightarrow D(A)=D(B)$. 

For instance, $D_\I$ as introduced in Definition \ref{def:density points for ideals} is lifted from $\Borel({}^\omega \baseset)/\I$ if $\I$ is shift invariant (see p.~\pageref{definition shift invariant ideal}). 
Then in fact $A\subseteq_\I B\Longrightarrow D_\I(A) \subseteq D_\I(B)$ by the definition of $D_\I$. 

Recall that the density property holds for $D$ and $\I$ if $D(A)=_\cI A$ for all Borel sets $A$ (see Definition \ref{definition density property}). 
We now give a useful condition for proving this for specific ideals. 
To state this condition, we say that $D$ is \emph{$\cI$-compatible} if $A \subseteq_\cI B \Longrightarrow D(A) \subseteq_\cI D(B)$ and $A \perp_\cI B \Longrightarrow D(A) \perp_\cI D(B)$ for all Borel sets $A$ and $B$. 
We further say that $D$ is \emph{$\cI$-positive} if $D(A)\cap A$ is $\cI$-positive for all $\cI$-positive Borel sets $A$. 


\begin{proposition} \label{equivalence to density property} 
If $D \colon \Borel({}^\omega \baseset) \to \Borel({}^\omega \baseset)$ is a function that is lifted from $\Borel({}^\omega \baseset)/\I$, then the following statements are equivalent. 
\begin{enumerate-(a)} 
\item 
$D$ is $\cI$-compatible and $\cI$-positive. 
\item 
The density property holds for $D$ and $\I$. 
\end{enumerate-(a)} 
\end{proposition} 
\begin{proof} 
  It is clear that the $\cI$-density property implies that $D$ is
  $\cI$-positive and $\cI$-compatible.  To see that these conditions
  imply the density property, take any Borel set $A$.  We aim to show
  that $D(A)=_\cI A$.

We first show that $B_0= A\setminus D(A)$ is in $\cI$. Towards a contradiction, assume that $B_0$ is $\cI$-positive. 
Then $D(B_0)\setminus D(A)$ is also $\cI$-positive, since it contains $D(B_0)\cap B_0$ as a subset, and the latter is $\cI$-positive since $D$ is $\cI$-positive. 
On the other hand, we have $D(B_0)\setminus D(A)\in \cI$ since $B_0\subseteq A$ and $D$ is $\cI$-compatible, contradiction. 

It remains to show that $B_1= D(A)\setminus A$ is in $\cI$. Assume
that $B_1$ is $\cI$-positive, so that in particular $D(A)$ is
$\cI$-positive.  The set $C=D(B_1)\cap B_1$ is $\cI$-positive, since
$D$ is $\cI$-positive.  We thus obtain $C\perp_\cI D(A)$, as
$C\subseteq D(B_1)$ and we have $D(B_1)\perp_\cI D(A)$ since 
$B_1 \perp_\cI A$ (in fact $B_1$ and $A$ are disjoint) and $D$ is
$\cI$-compatible.  However, this contradicts the fact that $C$ is $\I$-positive and 
$C \subseteq B_1 \subseteq D(A)$.
\end{proof}

\np 


\subsection{The density property for strongly linked tree ideals}
\label{section strongly linked tree forcings} 

To obtain the density property for $D_{\I_\PP}$, 
we will make two modest assumptions on $\PP$. 
Let $\bK_\I$ denote a fixed subset of $\PP$ coding $\bL_\I$ (from Convention~\ref{convention}) in the sense that $\bL_\I=\{[T]\mid T\in \bK_\I\}$. 

\begin{definition} \label{definition stem property} 
We say $(\PP,\bK_\I)$ has the \emph{stem property} 
  if for all $T\in\PP$ and
  $\cIs_\PP$-almost all $x\in [T]$, there are infinitely many $n\in\omega$ such that
  there is some $S \leq T$ with $x \in [S]$ and
  $S/(x{\upharpoonright}n)\in \bK_\I$. 
Since $\bK_\I$ is fixed for each $\PP$, we may also just say $\PP$ has the stem property.
\end{definition} 

The condition is trivially true for all $x\in [T]$ when $\bK_\I=\{T\in\PP\mid \stem(T)=\emptyset\}$ and we only introduce it to deal with the case of random forcing. 
For random forcing, recall $\bK_\I=\{T\in\PP\mid \mu([T])>\frac{1}{2}\}$. 
Then the stem property follows from Theorem~\ref{construction of density points} or Lebesgue's density theorem. 

\medskip
The next lemma shows that $D_\I$ is $\I$-compatible for $\I=\cIs_\PP$ provided that $D_\I$ is $\I$-positive and $\PP$ has the stem property.

\begin{lemma} \label{lem:disjointnessproperty} 
Assume that $\PP$ is a collection of trees with the stem property, all Borel sets are $\PP$-measurable, $\I=\cIs_\PP$ and
$D_\I([T])\cap[T] \notin \I$ for all $T\in\PP$. 
Then $D_\I(A)\cap D_\I(B) =_\I D_\I(A \cap B)$ for all Borel sets $A$, $B$. 
In particular, $D_\I$ is $\I$-compatible.
\end{lemma} 
\begin{proof} 
It is easy to see that $D_\I(A \cap B) \subseteq D_\I(A)\cap D_\I(B)$,
so suppose towards a contradiction 
that there are Borel sets $A$, $B$ with $C=(D_\I(A)\cap D_\I(B)) \setminus D_\I(A \cap B) \notin \I$.
Since $C$ is 
Borel (it is a Boolean combination of ${\bf \Sigma}^0_2$ sets) and hence $\PP$-measurable, by Lemma \ref{characterization of positive sets} there is some $S_0\in \PP$ with $[S_0]\subseteq_\I C$. 

Since $[S_0] \subseteq_\I D_\I(A)$ we can pick 
$x\in[S_0]\cap D_\I(A)$ witnessing the stem property for $S_0$. Let $n_x \in \omega$ witness that
$x \in D_\I(A)$. 
By the choice of $x$, there is some $m\geq n_x$
and $S_1\leq S_0$ with $S_1/(x{\upharpoonright}m) \in \bK_\I$. 
Since $m\geq n_x$, $[S_1]\cap A\notin \I$. 
By Lemma \ref{characterization of positive sets}, there is some $S_2\in \PP$ with $[S_2]\subseteq_\I [S_1]\cap A$. 

Thus $[S_2] \subseteq_\I C \subseteq D_\I(B)$. Repeating the previous argument with $A$ replaced by $B$ yields some $S_3\in\PP$ with $[S_3] \subseteq_\I [S_2]\cap B$. 

Since $[S_3]\subseteq_\I A \cap B$, we have 
$D_\I([S_3])\subseteq D_\I(A \cap B)$. 
Since $[S_3]\subseteq_\I C$ it follows that $D_\I([S_3])\cap [S_3]\subseteq_\I C \cap D_\I(A \cap B) = \emptyset$. 
This contradicts the assumption that $D_\I([S_3])\cap [S_3] \notin \I$.
\end{proof} 


To obtain $\I$-positivity we assume the following property (already mentioned in Section~\ref{s.intro}). 

\begin{definition} \label{definition strongly linked} 
A collection of trees $\PP$ is called \emph{strongly linked} if 
any $S,T\in \PP$ with $\stem(S) \sqsubseteq \stem(T)$ and $\stem(T)\in S$ are compatible in $\PP$. 
\end{definition} 

This condition holds for all ccc tree forcings that we study in this paper except random forcing. 
For instance for Hechler forcing, eventually different forcing, Laver forcing $\Laver[F]$ with a filter and Mathias forcing $\Mathias[F]$ with a shift invariant filter. 
Clearly the condition implies that $\PP$ is 
$\sigma$-linked and ccc. 
Thus $\I_\PP=\cIs_\PP$ by Section \ref{section Positive Borel sets}.

\begin{remark}\label{lem:Random forcing no strongly linked dense subset} 
The null ideal $\I_\mu$ is not a strongly linked tree ideal.
For suppose that $\PP$ is a strongly linked tree forcing with $\I_\PP=\I_\mu$.
Fix a nowhere dense closed set $C$ of positive measure and find $S\in\PP$ with 
$[S]\subseteq_\mu C$ by Lemma~\ref{characterization of positive sets}. 
We write $A\subseteq_\mu B$ if $\mu(A\setminus B)=0$. 
Find some $t\in S$ with $\stem(S)\sqsubseteq t$ and $\frac{\mu([S]\cap N_t)}{\mu(N_t)}<1$; 
such a $t\in S$ exists since $C$ is nowhere dense and $[S]\subseteq_\mu C$. 
Let further $A=N_t\setminus [S]$. Since $\frac{\mu(A\cap
  N_t)}{\mu(N_t)}>0$, 
  there is some $T\in\PP$ with $[T]\subseteq_\mu A$ (again by Lemma~\ref{characterization of positive sets}). 
Then $\stem(S)\sqsubseteq t\sqsubseteq \stem(T)$, but $S$ and $T$ are incompatible in $\PP$. 
\end{remark} 

\begin{lemma} \label{every branch in T in a density point} 
Assume that $\PP$ is a strongly linked tree forcing---whence by Convention~\ref{convention} $\bL_{\I_\PP}=\{[T]\in\PP\mid \stem(T)=\emptyset\}$. Let $\I=\cI_\PP$. Then $D_\I$ is $\I$-positive. In fact for any $T\in \PP$, $D_\I([T])=[T]$. 
\end{lemma}
\begin{proof} 
It follows from the definition that $D_\I([T])\subseteq [T]$. To see the other inclusion,
let $x \in [T]$ be given; we show that $x$ is an $\I$-density point of $[T]$. 
Let $m\in\omega$ be such that 
$\stem(T) = x{\upharpoonright}m$. 
Suppose we are given $S\in\PP$ such that $\stem(S) = x{\upharpoonright}n$ for some $n\geq m$. 
Since the stems of $S$ and $T$ are compatible, $\stem(S) \in T$, and $\PP$ is strongly linked, $S$ and $T$ are compatible. 
Therefore $[S]\cap [T]$ is an $\I$-positive set as required. 
\end{proof} 

Since $\I_\PP =\cIs_\PP$ when $\PP$ is ccc, Proposition \ref{equivalence to density property} together with Lemmas \ref{lem:disjointnessproperty} and \ref{every branch in T in a density point} imply the following version of Lebesgue's density theorem, viz.\ Theorem~\ref{intro theorem D_I}:

\begin{corollary} \label{density property for concrete forcings}
For any strongly linked tree forcing $\PP$, the $I_\PP$-shift density property (cf.\ Definition \ref{def:density points for ideals}) holds. 
In particular, the $\I_\PP$-shift density property holds when $\PP$ is Cohen forcing $\Cohen$, Hechler forcing $\Hechler$, eventually different forcing $\mathbb{E}$, Laver forcing $\Laver[F]$ with a filter and Mathias forcing $\Mathias[F]$ with a shift invariant filter (cf.\ Section \ref{section list of tree forcings}). 
\end{corollary} 

For random forcing, the $\I_\mu$-shift density property follows from Lemma~\ref{lem:separating Lebesgue density} and Lebesgue's density theorem. A self-contained proof is obtained from the fact that $D_\I$ is $\I$-positive by Theorem~\ref{construction of density points} together with Proposition \ref{equivalence to density property} and Lemma \ref{lem:disjointnessproperty}.


Note that if $\PP$ is topological and $\cI_\PP$ has the density property, then one can describe $D_{\I_\PP}(A)$ as follows using $\tau_\PP$. 
First note that $\cI_\PP=\cIs_\PP$ by the proof of \cite[Lemma 3.8]{FKK16}. 
Moreover, for any Borel set $A$ there is some $\tau_\PP$-open set $U$ with $A\triangle U$ $\tau_\PP$-meager by the $\tau_\PP$-Baire property. 
Since $\cI_\PP$ equals the set of $\tau_\PP$-meager sets, 
$D_{\I_\PP}(A)$ is almost equal to a $\tau_\PP$-open set. However, note that even for ccc collections it is not clear how to find such a set in a simply definable way. 

We conclude this section with some observations about other variants of the Definition \ref{def:density points for ideals} of density points and the shift density property. 

\begin{remark}\label{remark ccc} 
If $\PP$ is ccc and the density property holds for all Borel sets, then it already holds for all $\PP$-measurable sets. 
This is the case because any $\PP$-measurable set equals $A\triangle B$ for some Borel set $A$ and some $B\in \cI_\PP$ by the ccc and Lemma \ref{measurability for the null ideal}. 
\end{remark}

\begin{remark} \ 
\label{rmk: failure for the Hechler ideal} 
  \begin{enumerate-(1)}
  \item If we let $\bL_\I=\I^+$ be the $\I$-positive sets for Hechler forcing $\Hechler$ and $\I=\I_{\Hechler}$, then the density property fails. 
To see this, let $T$ be the Hechler tree with empty stem given by the constant function with value $1$. 
With this variant $[T]$ does not satisfy the density property, since
no $x\in[T]$ is an $\I$-shift density point of $[T]$, as witnessed by
$N_{\langle0\rangle}\in\bL_\I$. 
\item \label{failure for the Hechler ideal} 
If $n_0$ depends on $B\in\bL_\I$ in Definition \ref{def:density points
  for ideals}, then the density property fails for $\I_{\Hechler}$ as well. 
To see this, let $x$ be the function with $x(n)=n+1$ for all $n\in\omega$ and let $T_{\emptyset, x}$ be the tree with empty stem given by $x$. 
Let further $A=\bigcup_{t\in {}^{<\omega}\omega}[T_{t^\smallfrown\langle|t|\rangle,t^\smallfrown\langle|t|\rangle^\smallfrown0^\infty}]$; this contains all $y\in{}^\omega\omega$ with $y(n)=n$ for some $n\in\omega$.  
It is sufficient to show $[T_{\emptyset,x}]\subseteq D(A)\setminus A$ by Lemma \ref{characterization of positive sets}. 

It is easy to see that $[T_{\emptyset,x}]\cap A=\emptyset$, since any $y\in[T_{\emptyset,x}]$ satisfies $y(n)\geq n+1$ for all $n\in\omega$. 

We now claim that $y\in D(A)$ for all $y\in{}^\omega\omega$. So take any $\I_{\Hechler}$-positive Borel set $B$. 
It is sufficient to assume that $B=[T_{s,u}]$ for some $s\in {}^{<\omega}\omega$ and $u\in{}^{\omega}\omega$ by Lemma \ref{characterization of positive sets}. 
Then for any $n\geq u(|s|$), $\sigma_{y{\upharpoonright}n}^{-1}([T_{s,u}/s)])\cap A=[(y{\upharpoonright}n)^\smallfrown (T_{s,u}/s)]\cap A$ contains $[T_{t,t^\smallfrown v}]$ for $t=y{\upharpoonright}n^\smallfrown\langle n\rangle$ and $v\in{}^\omega\omega$ with $v(i)=u(|s|+i+1)$ for all $i\in\omega$. 
Hence it is $\I_{\Hechler}$-positive and thus $y$ is a density point of $A$. 
  \end{enumerate-(1)}

\end{remark}

\np 


\subsection{Complexity of the density operator} \label{section complexity} 


In this section, we give an upper bound for the complexity of the operator $D_\I$ for relevant cases of tree ideals $\I$ of the form $\I_\PP$. 
$\bK_\I$ is fixed as in the previous section. 

Recall that a subset $A$ of ${}^\omega\omega$ is called 
\emph{universally Baire} if for any topological space $Y$ and any continuous function $f\colon Y\rightarrow{}^\omega\omega$, $f^{-1}(A)$ has the property of Baire. 
A function
$f\colon A\rightarrow B$ between subsets of ${}^\omega\omega$ is called \emph{universally Baire
  measurable} if $f^{-1}(U)$ is universally Baire for every
relatively open subset $U$ of $B$.  
Note that all universally Baire sets have the Baire property, are Lebesgue measurable and Ramsey \cite[Theorem 2.2]{MR1233821}.
A definable forcing $\PP$ is called \emph{absolutely ccc} if the ccc holds in all generic extensions. A $\Delta^1_2$ predicate is \emph{absolutely $\Delta^1_2$} if its $\Sigma^1_2$ and $\Pi^1_2$ definitions are equivalent in all generic extensions.

%
%
%
%

\begin{lemma} \label{complexity of density operator} 
Suppose that $\I = \I_\PP$, where $\PP$, $\leq_\PP$ and $\perp_\PP$ are $\Sigma^1_1$, $\bK_\I$ is a $\Sigma^1_1$ subset of $\PP$ and $\PP$ is absolutely ccc. 
Then 
\[
D_\I \colon \Borel({}^\omega \baseset) \to \Borel({}^\omega \baseset)
\]
is induced by a $\Delta^1_2$ function from Borel codes to Borel codes. 
Moreover, this function is universally Baire measurable.
\end{lemma} 
\begin{proof} 
Recall that we write $\Borelcodes$ for the $\Pi^1_1$ set of Borel codes and $\Borelset_x$ the set coded by a Borel code $x$. 
Let $\varphi(x)$ denote the formula $x\in \Borelcodes\ \&\ \forall T \in \bK_\I\ \Borelset_x \cap [T]\notin \I$. 
Let $\psi(x)$ denote the following statement: 
$x\in \Borelcodes$ and there is some $m\leq\omega$ and a sequence $\vec{T}=\langle  T_i\mid i<m\rangle$ from $\PP$ with $\forall i<m\ [T_i]\subseteq_\I \Borelset_x$ and $\forall T \in L_\I\ \exists i<m \ T \parallel_{\PP} T_i$.

\begin{claim} 
$\forall x\ \varphi(x)\Longleftrightarrow \psi(x)$. 
\end{claim} 
\begin{proof} 

If $\varphi(x)$ holds, inductively define an antichain $\vec{T}=\langle T_\xi \mid \xi < \theta \rangle$ from $\PP$ with $[T_\xi]\subseteq_{I_\PP} \Borelset_x$ for all $\xi < \theta$. 
Suppose that $\vec T^{(\alpha)} = \langle T_\xi \mid \xi < \alpha \rangle$ is already defined. 
If there is $T\in \bL_\I$ which is incompatible with each element of $\vec T^{(\alpha)}$, we may find $T_\alpha \in \PP$ with $T_\alpha \subseteq T$ and $[T_\alpha] \subseteq \Borelset_x$ since $\phi(x)$ holds. 
We can thus extend the antichain by adding $T_\alpha$. 
By the ccc we must reach some $\theta < \omega_1$ such that each $T \in {\bL}_\I$ is compatible to an element of $\vec{T}=\langle T_\xi \mid \xi < \theta \rangle$. 
Enumerate $\vec{T}$ in order type $m\leq \omega$ to obtain a witness to $\psi(x)$. 

If conversely $\psi(x)$ holds, then for any $T \in \bK_\I$ we may find some $i\in\omega$ with $T_i \parallel T$ and $[T_i] \subseteq_{\I} \Borelset_x$, so one can infer $\Borelset_x\cap [T] \notin \I$ from Lemma \ref{characterization of positive sets}. Thus $\varphi(x)$ holds. 
\end{proof}  

We now check that $\varphi(x)$ is a $\Pi^1_2$ and $\psi(x)$ a $\Sigma^1_2$ formula. 
First note that the statement $\Borelset_x\in\cN_\PP$ is $\Sigma^1_2$, since this holds if and only if there is a (countable) maximal antichain $\vec{S}=\langle S_i\mid i<m \rangle$ in $\PP$ with $\Borelset_x\cap \bigcup_{i<m} [S_i]=\emptyset$. 
Since $\I = \cI_\PP$ is the $\sigma$-ideal generated by $\cN_\PP$, the statements $\Borelset_x\in\I$ and $[T]\subseteq_\I \Borelset_x$ are $\Sigma^1_2$ as well. 
Thus $\varphi(x)$ and $\psi(x)$ are indeed of said complexity.


Since $\PP$ is absolutely ccc, the argument above shows that $\forall x\ \Phi(x)\Longleftrightarrow \Psi(x)$ is absolute to generic extensions. 
The above easily shows that $S=\{(x,s) \in \Borelcodes \times {}^{<\omega} \baseset \mid \forall T\in \bK_\I \ \sigma_s([T])\cap \Borelset_x\notin\I\}$ is absolutely $\Delta^1_2$, where $\Borelcodes$ denotes the set of Borel codes. 
Thus the function $\hat D$ which sends each Borel code $x$ to a Borel code $\hat D(x)$ for $D_\I(\Borelset_x)$ has a $\Sigma^1_2$ graph. 
Finally for each $s \in {}^{<\omega}$, $\{x \mid (x,s) \in S\}$ is absolutely $\Delta^1_2$ and hence universally Baire by \cite[Theorem 2.1]{MR1233821}.
It follows easily that $\hat D$ is universally Baire measurable.
\end{proof}

We want to point out that  the previous lemma remains true with virtually the same proof if we replace \emph{absolutely ccc} and  \emph{absolutely $\Delta^1_2$} by \emph{provably ccc} and \emph{provably $\Delta^1_2$}. 

\medskip
We now show that for all strongly linked tree forcings listed in Section \ref{section list of tree forcings}, the density operator $D_{\I_\PP}$ is induced by a universally Baire measurable function. 
Recall that a forcing is called \emph{Suslin} if $\PP$, $\leq_\PP$ and $\perp_\PP$ are $\Sigma^1_1$. 

\begin{proposition} 
For any strongly linked Suslin tree forcing $\PP$, $D_{\I_\PP}$ is induced by a universally Baire measurable function. 
\end{proposition} 
\begin{proof} 
If $\PP$ is strongly linked, then $S \parallel T$ if and only if $\stem(S) \in T$, $\stem(T) \in S$, and $\stem(S)$ and $\stem(T)$ are compatible, so $\perp_\PP$ is arithmetical and hence $\Sigma^1_1$. 
Moreover, the fact that a Suslin tree forcing is strongly linked is $\Pi^1_2$ and hence absolute. Thus $D_{\I_\PP}$ is induced by a $\Sigma^1_2$, universally Baire measurable function on the Borel codes by Lemma \ref{complexity of density operator}. 
\end{proof}

\np 

\section{Ideals without density} 

In this section, we study various counterexamples to density properties. 

\subsection{Counterexamples} \label{section counterexamples} 

We first give counterexamples to the density property in Definition \ref{def:density points for ideals} for several non-ccc tree forcings. 

\begin{proposition} \label{counterexamples for Mathias, Sacks, Silver} 
  Let $\Mathias$, $\Silver$, $\Sacks$ denote Mathias, Silver and Sacks
  forcing. Then $\cI_{\Mathias}$, $\cI_{\Silver}$ and $\cI_{\Sacks}$
  do not have the shift density property. 
\end{proposition}

\begin{proof}
  To see that $\cI_{\Mathias}$ does not have the $\cI_{\Mathias}$-shift density property, let
  $A = \{ f \in {}^\omega2 \st f(2n+1) = 1 \text{ for all } n \in \omega
  \}$. Note that $A = [S]$ for some $S \in \Mathias$ and hence
  $A \notin \cI_{\Mathias}$. We aim to show that no $x \in A$ is an
  $\cI_{\Mathias}$-density point of $A$, i.e.
  $A \cap D_{\Mathias}(A) = \emptyset$. Then in particular
  $A \triangle D_{\Mathias}(A) = A \cup D_{\Mathias}(A) \notin
  \cI_{\Mathias}$.

  Let $x \in A$ be arbitrary and let $T \in \Mathias$ be a perfect
  tree such that $\spl(T) = 2\mathbb{N}$ and
  $t^\smallfrown \langle i\rangle ^\smallfrown \langle j\rangle \in T$ iff $j = 0$ for all
  $t \in \spl(T)$ and $i,j \in \{0,1\}$. In particular
  $\stem(T) = \emptyset$. Let $n_0 \in \omega$ be arbitrary and let
  $n \geq n_0$ be even. Then
  $f_{x \upharpoonright n}[T] \cap A = \emptyset$ and thus $x$ is not
  an $\cI_{\Mathias}$-density point of $A$. 
  
  As $\Mathias \subseteq \Silver \subseteq \Sacks$, the claim also
  holds for $\cI_{\Silver}$ and $\cI_{\Sacks}$.  
\end{proof}



The following is a similar counterexample for Laver and Miller forcing. 

\begin{proposition} \label{counterexamples for Laver, Miller} 
  Let $\Laver$, $\Miller$ denote Laver and Miller forcing. Then $\cI_{\Laver}$ and $\cI_{\Miller}$ do not
  have the shift density property.
\end{proposition}

\begin{proof}
  Let
  $A = \{ f \in {}^\omega\omega \st f(n) \text{ is even for all } n \in
  \omega \}$. Then $A = [S]$ for some $S \in \Laver$ so in particular
  $A \notin \cI_{\Laver}$. We aim to show that no $x \in A$ is an
  $\cI_{\Laver}$-density point of $A$, i.e.
  $A \cap D_{\Laver}(A) = \emptyset$. As above this implies
  $A \triangle D_{\Laver}(A) \notin \cI_{\Laver}$.

  Let $x \in A$ be arbitrary and let $T \in \Laver$ be a perfect tree
  such that $\stem(T) = \emptyset$ and
  $[T] = \{ g \in {}^\omega\omega \st g(n) \text{ is odd for all } n \in
  \omega \}$. Let $n_0 \in \omega$ be arbitrary and let $n \geq
  n_0$. Then $f_{x \upharpoonright n}[T] \cap A = \emptyset$ and thus
  $x$ is not an $\cI_{\Laver}$-density point of $A$.
  
  Since $\Laver \subseteq \Miller$, the claim for  $\cI_{\Miller}$ follows. 
\end{proof}




\np 


\subsection{Selectors for the ideal of countable sets} \label{section: no selector}
We now study the ideal $\cI$ of countable sets. In contrast to the previous results, we will show that 
there is no Baire measurable selector with Borel values for the equivalence relation of having countable symmetric difference on the set of Borel subsets. 
This implies that the density property fails for $\cI$ for any reasonable notion of density point. 

To state the result formally, we need the following notions. 
A \emph{selector} for an equivalence relation $E$ is a function that chooses an element from each equivalence class. 
We generalize this notion by replacing equality with a subequivalence relation $E'$ of $E$. 

\begin{definition} \label{definition of selector 1} 
Suppose that $E'\subseteq E$ are equivalence relations on a set $B$ and $A\subseteq B$. 
A \emph{selector for $E/E'$ on $A$} is an 
$(E,E')$-homomorphism $A\rightarrow B$ that uniformizes $E$. 
\end{definition} 

Equivalently, the induced map on $B/E'$ 
is a selector for the equivalence relation on $B/E'$ induced by $E$. 

In the following, $E'$ will be equality of decoded sets, $E$ the equivalence relation of having countable symmetric difference, $A$ will be the set of $F_\sigma$-codes, denoted also by $\Borelcodes_{F_\sigma}$,  and $B=\Borelcodes$, the set of Borel codes. 
More precisely, 
an $F_\sigma$-code is a sequence 
$\vec{T}=\langle T_n\mid n\in\omega\rangle$, where $T_n$ is a subtree of ${}^{<\omega}2$ for each $n\in\omega$. 
Note that $\Borelcodes_{F_\sigma}$ is therefore a Polish space. 
We can assume that $\Borelcodes_{F_\sigma}\subseteq\Borelcodes$. 

Equality of Borel sets and equality modulo $\I$ induces the following equivalence relations on $\Borelcodes$. 
As before, let $\Borelset_x$ denote the Borel set coded by $x\in\Borelcodes$. 
Let $E_=$ denote the equivalence relation on $\Borelcodes$ of equality of decoded sets, i.e. $(x,y)\in E_= \Longleftrightarrow \Borelset_x=\Borelset_y$. 
Let further $(x,y)\in E_\I\Longleftrightarrow \Borelset_x\triangle \Borelset_y\in \cI$ for $x,y\in \Borelcodes$. 

\begin{definition} \label{definition of selector 2} 
A \emph{selector for $\I$ with Borel values} is a selector  $\Borelcodes_{F_\sigma}\to\Borelcodes$ for $E_\I/E_=$. 
\end{definition} 

The restriction to $F_\sigma$-codes is purely for a technical reason: the proof of Theorem \ref{no selector for ideal of countable sets} will show that there is no reasonably definable selector for $\I$ on $\Borelcodes_{F_\sigma}$. 
It follows that there is no such selector on the set of all Borel codes. 

The motivation for this definition is as follows. Consider any notion
of density points for $\I$ with Borel values, i.e. such that for any
Borel set $A$, the set of density points of $A$ is Borel. 
If the density property holds for this notion, then the density operator induces a selector for $\I$ on the set of Borel codes. 


\begin{theorem} \label{no selector for ideal of countable sets}
There is no Baire measurable selector for $\I$, the ideal of countable sets, with Borel values. 
\end{theorem}

It follows that there is no selector $\Borel({}^\omega \baseset)\to\Borel({}^\omega \baseset)$ for $=_\I$ 
that is induced by a universally Baire measurable function (on the codes). 

Note that this result is analogous to the fact that $E_0$ does not have
a Baire measurable selector. 
However, the rather short proof of the latter (see e.g. \cite[Example 1.6.2]{Hjorth10}) is of no use here. 

The current proof simplifies that of a previous version, and resembles Kanovei's version \cite{kanovei-no-selector} in some aspects, but was arrived at independently. 


\medskip

Before we begin the long proof of Theorem~\ref{no selector for ideal of countable sets} let us sketch the main ideas.
Suppose that $F\colon\Borelcodes_{F_\sigma}\to\Borelcodes$ is a selector for $E_I / E_=$. 
Thus, given a code $C$ for an $F_\sigma$ set, $\Borelset_{F(C)}$ is a representative for the equivalence class of $\Borelset_C$ modulo $I$.

We first construct a sequence of names for perfect trees $\langle \dot T_n \mid n\in\omega\rangle$ added by 
Cohen forcing, $\Cohen$.
Thus it is forced that
$\bigcup_{n\in\omega}[\dot T_n]$ is an $F_\sigma$ set; let $\dot C$ be a name for its code (i.e., for the sequence of trees).
We will fix a large enough countable elementary submodel $M$ of $H_{\omega_1}$ and 
carefully construct a name $\sigma \in M$ for an element of the representative $\Borelset_{F(\dot C)}$ of $\bigcup_{n\in\omega}[\dot T_n]$.

Now let $g\times h$ be any $\Cohen^2$-generic over $M$.  
We will find a different $g' \in M[g][h]$ which is also Cohen-generic over $M$ such that $\bigcup h = \sigma^{g'}$ and
\[
\bigcup_{n\in\omega}[\dot T^{g}_n] = \Big(\bigcup_{n\in\omega}[\dot T^{g'}_n] \Big)\setminus \{\sigma^{g'}\}.
\] 
Thus letting $C=\langle \dot T^{g}_n \mid n\in\omega\rangle$, a code for $\bigcup_{n\in\omega}[\dot T^{g}_n]$, 
and $C' =\langle \dot T^{g'}_n \mid n\in\omega\rangle$, a code for $\bigcup_{n\in\omega}[\dot T^{g'}_n]$, 
it follows that $\Borelset_{F(C)} = \Borelset_{F(C')}$, i.e., the two sets have the same representative.
A careful choice of the name $\sigma$ and of the model $M$ is crucial for the construction of $g'$.
Since $\bigcup h = \sigma^{g'} \in \Borelset_{F(C')} = \Borelset_{F(C)} $,  
we will be able to conclude that  any Cohen real over $M[g]$ is an element of 
$\Borelset_{F(C)} \setminus \bigcup_{n\in\omega}[\dot T^g_n]$.
Since there are uncountably many Cohen reals over $M[g]$ this contradicts the assumption that $F$ is a selector.

\medskip

We now introduce some notation to aid us in the proof of Theorem~\ref{no selector for ideal of countable sets}.
To obtain our sequence $\langle \dot T_n\mid n<\omega\rangle$ we will use a particular presentation of Cohen forcing.
\begin{notation}\label{notation}
Let $\TT$ be the forcing whose conditions are finite subtrees $t$ of ${}^{<\omega}2$, 
ordered by end extension. 
By this we mean that one can only extend a tree at its maximal nodes. 
For any $\TT$-generic filter $G$ we refer to the tree $\bigcup G$ as the \emph{tree added by $\TT$}. 

Let $\TT^\omega=\prod_{n\in\omega} \TT_n$ denote the finite support product of $\omega$ copies of $\TT$ and $\dot{T}_n$ a $\TT_n$-name for the tree added by $\TT_n$. 
We can identify each $\dot{T}_n$ canonically with a $\TT^\omega$-name. 
\end{notation}

Since $\TT^\omega$ is countable and atomless, indeed it is a presentation of Cohen forcing.

We are now ready to prove Theorem \ref{no selector for ideal of countable sets}.

\begin{proof}[Proof of Theorem \ref{no selector for ideal of countable sets}] 
Suppose that there is a selector $F$ on $\Borelcodes_{F_\sigma}$ as in the statement of the theorem. 
Since $F$ is Baire measurable, there is a comeager $G_\delta$ subset
$A$ of the Polish space $\Borelcodes_{F_\sigma}$ such that $F{\upharpoonright}A$ is continuous. 
Let $x_A$ be a real in which both a Borel code for $A$ and a code for $F{\upharpoonright}A$ are computable. 
Moreover, let $M\prec H_{\omega_1}$ be countable with $x_A\in M$.

Work in $M$. 
Note that $\TT^\omega \in M$; the same holds for all other objects defined in \ref{notation}.
Let $\dot{C}\in M$ be a $\TT^\omega$-name for the canonical $F_\sigma$-code for $\bigcup_{n\in\omega} [\dot{T}_n]$ and $\dot{F}\in M$ a $\TT^\omega$-name for $F{\upharpoonright}A$. 

Working in $V$ again, $\dot{F}^g(x)=F(x)$ for all $x\in A\cap M[g]$ and all $\TT^\omega$-generic filters $g\in V$ over $M$. 


\begin{claim} 
For all $n\in\omega$, $1 \Vdash^M_{\TT^\omega} [\dot{T}_n]\setminus \Borelset_{\dot{F}(\dot{C})}$ is countable. 
\end{claim} 
\begin{proof} 
Assume towards a contradiction that $p\Vdash^M_{\TT^\omega} [\dot{T}_n]\setminus \Borelset_{\dot{F}(\dot{C})}$ is uncountable for some $p\in\TT^\omega$ and some $n\in\omega$. 
Since $A$ is comeager, any $\TT^\omega$-generic filter $g$ over $M$ with $p\in g$ will satisfy $\dot{C}^g\in A$. 
Let $C=\dot{C}^g$. 
Then 
$\dot{F}^g(\dot{C}^g)=F(C)$. 
Since $p\in g$, $M[g]\models [\dot{T}_n^g]\setminus \Borelset_{F(C)}$ is uncountable. 

Work in $M[g]$. 
By the perfect set property for Borel sets, there is a perfect tree $S$ whose branches are all elements of $[\dot{T}_n^g]\setminus \Borelset_{F(C)}$. 

By $\Pi^1_1$-absoluteness between $M[g]$ and $V$, all branches of $S$ in $V$ are elements of  $[\dot{T}_n^g]\setminus \Borelset_{F(C)}$. 
But this contradicts the assumption that $F$ is a selector. 
\end{proof}

We now resume working in $M$. 
For the construction of $ \sigma$ we need some more notation:
\begin{notation}\label{notation-pi}~
\begin{enumerate-(1)}
\item If $S$ is a subtree of ${}^{<\omega}2$, 
let $\operatorname{split}(S)$ denote the set of splitting nodes of $S$, i.e. those with at least two direct successors and let $\operatorname{term}(S)$ denote the set of terminal (i.e. maximal) nodes in $S$. 
\item\label{pi_t} 
If $S$ is an arbitrary subtree of ${}^{<\omega}2$, let $\pi_S$ denote the unique order isomorphism from a subtree of ${}^{<\omega}2$ onto $\operatorname{split}(S)\cup\operatorname{term}(S)$ 
such that every node in $\dom(\pi_S)$ is either splitting or maximal in $\dom(\pi_S)$. 
(If $S$ is perfect, then $\pi_S$ is the unique order isomorphism from ${}^{<\omega}2$ to $\operatorname{split}(S)$.) 
\item For each $x\in {}^\omega 2$, let $ \sigma_x$ be a $\TT^\omega$-name for $\bigcup_{n\in\omega} \pi_{\dot{T}_0}(x{\upharpoonright}n)$. 
\end{enumerate-(1)}
\end{notation}

Using the previous claim, elementarity, the fact that $\TT^\omega$ preserves cardinality and $({}^\omega 2)^M$ is uncountable in $M$, 
there is $x\in ({}^{\omega}2)^M$ and $p\in \TT^\omega$ with $p \Vdash^M_{\TT^\omega} \sigma_x\in [\dot{T}_0]\cap \Borelset_{\dot{F}(\dot{C})}$. 

\begin{notation}
Until the end of the proof of Theorem~\ref{no selector for ideal of countable sets} and Claim~\ref{cl.isomorphism}, let us fix $x\in ({}^{\omega}2)^M$ and $p_0\in \TT^\omega$ with $p_0 \Vdash^M_{\TT^\omega} \sigma_x\in [\dot{T}_0]\cap \Borelset_{\dot{F}(\dot{C})}$. 
Moreover, we assume $p_0=\emptyset$ for notational convenience. 
We shall also write $\sigma$ for $\sigma_x$ and $\stem(T)$ for $\stem(T)$. 
\end{notation}


Our next goal is to demonstrate that for any $\TT^\omega$-generic $g$ over $M$, $\Borelset_{F(\dot C^g)}\setminus \Borelset_{\dot C^g}$ must contain every real which is Cohen over $M[g]$ (and hence must be uncountable, leading to a contradiction). 
The next claim will be crucial.

\begin{claim}\label{cl.isomorphism}
In every $\TT^\omega\times\Cohen$-generic extension $M[g \times h]$ we can find  $g'$ which is $\TT^\omega$-generic over $M$ such that
\begin{align}
\label{e.iso.cohen}  \sigma^{g'} &= \bigcup h, \\
\label{e.iso.trees} \Big(\bigcup_{n\in\omega} [\dot T^{g'}_n] \Big)\setminus \{\sigma^{g'}\} &= \bigcup_{n\in\omega} [\dot T^{g}_n].
\end{align}
In fact, there exist in $M$ a dense subset $D$ of $\TT^\omega\times\Cohen$ and a projection
\[
\pi\colon D \to  \TT^\omega
\]
so that 
the above holds with $g'= \pi\big[(g \times h )\cap D\big]$. (See \cite[Definition 5.2]{MR2768691} for the definition of projections.) 
\end{claim}

Assuming Claim~\ref{cl.isomorphism} for the moment, we can easily finish the proof of Theorem~\ref{no selector for ideal of countable sets} as follows. 
Fix any $\TT^\omega$-generic filter $g$ over $M$ in $V$.
Since $A$ is comeager and is coded in $M$,  we have $C=\dot{C}^g\in A$.  

\begin{claim}
$\Borelset_{F(C)}\setminus \Borelset_C$ contains every Cohen real over $M[g]$ in $V$. 
\end{claim} 
\begin{proof} 
Let $h$ be any Cohen generic filter over $M[g]$ in $V$ and let $c=\bigcup h$. 
One easily obtains $c\notin \Borelset_C$ by a density argument. 

It remains to show that $c\in \Borelset_{F(C)}$. 
By the previous claim we can find $g'\in M[g\times h]$ which is $\TT^\omega$-generic over $M$ and so that $c=\sigma^{g'}$ and $\Borelset_{\dot{C}^{g'}}=\Borelset_{C}\cup\{\sigma^{g'}\}$. 
Moreover $\sigma^{g'}\in \Borelset_{F(\dot{C}^{g'})}$ by the choice of $\sigma$ and Borel absoluteness between $M[g']$ and $V$. 
Since $F$ is a selector, $\Borelset_{F(\dot{C}^{g'})}=\Borelset_{F(C)}$. 
We thus have $c=\sigma^{g'}\in \Borelset_{F(\dot{C}^{g'})} = \Borelset_{F(C)}$. 
\end{proof} 

By the previous claim, $\Borelset_{F(C)}\setminus \Borelset_C$ is uncountable. 
But this contradicts our assumption that $F$ is a selector. 

\medskip

To finish the proof of Theorem~\ref{no selector for ideal of countable sets}, it remains to prove Claim~\ref{cl.isomorphism}.
We need some more notation:
\begin{notation}
For $s,t\in{}^{<\omega}2$ write $s \wedge t$ for the longest common initial segment of $s$ and $t$. 
\end{notation}

\begin{proof}[Proof of Claim~\ref{cl.isomorphism}.] 
Before we define 
$\pi$, we define a map 
$\hat{\pi}$ on a larger domain. The projection $\pi$ will be the restriction of $\hat{\pi}$ to a dense subset $D$ of $\TT^\omega\times\Cohen$. 
This approach is convenient since $\hat{\pi}$ captures both the action of $\pi$ on conditions in $\TT^\omega\times\Cohen$, 
as well the ``continuous extension'' of this action to generic objects for $\TT^\omega\times\Cohen$.

The domain of $\hat{\pi}$ is the set of pairs $(\langle t_n\mid n < l \rangle, p)$ such that 
$l\in\omega+1$, $p \in {}^{\leq \omega} 2$,  and each $t_n$ is a subtree of ${}^{<\omega} 2$.
Suppose we are given $q=(\langle t_n \mid n < l \rangle, p)\in \dom(\hat{\pi})$. 
We shall let
\[
\hat{\pi}(q)=   \langle t'_j \mid j< l'\rangle, 
\] 
where $\langle t'_j \mid j< l'\rangle$ is a sequence of subtrees of ${}^{<\omega}2$ with $l' \in \omega+1$ which is constructed as follows. 


We first construct a sequence 
$\langle n(i) \mid i < m\rangle$ of natural numbers $n(i)<l$, where $m\in\omega+1$.
Let $k \in \omega$ and suppose we have already defined $\langle n(i) \mid i<k\rangle$ (if $k=0$, this is the empty sequence). 
Let $n(k)$ be least number, if it exists, that satisfies the following conditions: 

\begin{conditions} \ 
\begin{enumerate-(a)}
\item
If $k>0$, then 
\begin{enumerate-(i)} 
\item 
$n(k) > n(k-1)$,
\item 
$p {\wedge} \stem(t_{n(k-1)}) \subsetneq p {\wedge} \stem(t_{n(k)})$ and 
\item  
$\stem(t_{n(k)}) \perp p$. 
\end{enumerate-(i)} 
\item 
Letting 
$t  =\{p{\upharpoonright}j\mid j\in\omega \} \cup \bigcup_{i\leq k} t_{n(i)}$, 
we have 
\begin{enumerate-(i)} 
\item 
$x{\res} (k+1)\in\dom(\pi_t)$ and\footnote{See Notation~\ref{notation-pi}\ref{pi_t} for the definition of the isomorphism $\pi_t$.}
\item 
$\pi_t\big(x{\res} (k+1)\big) \subseteq p$. 
\end{enumerate-(i)} 
\end{enumerate-(a)}
\end{conditions} 
%


%
If such $n(k)$ does not exist, let $m=k$ and declare the construction of $\langle n(i) \mid i < m\rangle$ complete.
If $n(k)$ is defined for each $k<\omega$ we let $m=\omega$.
This completes the construction of 
$\langle n(i) \mid i < m\rangle$.

We now define $\langle t'_j \mid j< l'\rangle$ and thus $\hat{\pi}$. 
First, let 
\begin{equation}\label{e.t'_0}
t'_0 = \{p {\res} j \mid j \in\omega \} \cup \bigcup_{i<m} t_{n(i)}. 
\end{equation}
Note that $m<\omega$ and $\pi_{t'_0}(x{\res} m) = p$ if $p\in \CC$ and in general, $\bigcup_{k<m} \pi_{t'_0}(x{\res}(k+1)) = p$. 
Further, let $\langle t'_j\mid 0<j<l'\rangle$ be the increasing re-enumeration of 
what remains of the sequence $\langle t_n \mid n<l\rangle$ after removing the subsequence $\langle t_{n(i)} \mid i < m\rangle$.
In other words, letting $\langle \bar n(i)\mid i< \bar{l}\rangle$ be the increasing enumeration of $l  \setminus \{n(i)\mid i<m\}$, we define $l' = 1 + \bar{l}$ and  $t'_{j+1} = t_{\bar n(j)}$ for $j<\bar{l}$. 
(Note that $l'=l-m+1$ if $l$ is finite.) 
If $(\langle t_n \mid n < l \rangle, p) \in \TT^\omega\times\Cohen$, then clearly $\langle t'_j\mid j<l'\rangle\in\TT^\omega$.

To define $D$, we introduce following notation. 
Suppose that in the context of the above construction of $\langle n(i) \mid i < m\rangle$, we have 
$(\langle t_n \mid n < l \rangle, p) \in \TT^\omega\times\Cohen$ (whence $m<\omega$)
and furthermore, $m>0$. 
Then we denote the last initial segment $p \wedge \stem(t_{n})$ of $p$ picked up in this construction as 
\begin{equation}\label{e.def.l}
s\big(\langle t_n \mid n < l \rangle, p\big) := p \wedge \stem(t_{n(m-1)}).
\end{equation}
When this notation is used in the definition of $D$ below, it is implied that $0<m<\omega$, so that $t_{n(m-1)}$ is defined.

Note that 
$\pi_t\big(x{\res} (k+1)\big) \perp \stem(t_{n(k)})$ and $\pi_t\big(x{\res} k\big) \subseteq \stem(t_{n(k)})$ for each $k<m$ and
$p \wedge \stem(t_{n(k')}) \subsetneq p \wedge \stem(t_{n(k)})\subsetneq p$ for each $k'<k<m$. 
Let 
\begin{multline}\label{e.def.D}
D = \big\{ (\langle t_n \mid n < l \rangle, p) \in \TT^\omega\times\Cohen \mid 
\lvert s\big(\langle t_n \mid n < l \rangle, p\big)\rvert=|p|-1\text{ and }\\
\forall n < l\; \operatorname{split}(t_n) \neq \emptyset \land \lvert p\wedge\stem(t_n)\rvert<|p| 
\big\},
\end{multline}
where $s\big(\langle t_n \mid n < l \rangle, p\big)$ is as defined in \eqref{e.def.l} above.
Clearly $D$ is dense in $\TT^\omega\times\Cohen$.
The reason for restricting the domain of $\hat{\pi}$ in this way is that for $q \in D$, it is guaranteed that the sequence 
$\langle n(i) \mid i < m\rangle$ grows in the intended manner when extending the $\TT^\omega$-component of $q$ (see below). 
\begin{subclaim}\label{scl.dense}
The map $\pi = \hat{\pi}{\upharpoonright} D$ is a projection to $\TT^\omega$.
\end{subclaim}
\begin{proof}
To see that $\pi$ is order-preserving, 
suppose we are given 
$q = \big(\langle t_j \mid j < l\rangle,p\big)$ and
$\tilde q = \big(\langle \tilde t_j \mid j < \tilde l\rangle,\tilde p\big)$ from $D$ with $\tilde q \leq q$. 
Write $\pi(q)=\langle t'_j \mid j < l'\rangle$ and $\pi(\tilde q)=\langle \tilde t'_j \mid j < \tilde l'\rangle$.
We must show that $l' \leq \tilde l'$ and $\tilde t'_j$ end-extends $t'_j$ for each $j< l'$.

Write $\langle n(i) \mid i<m\rangle$ for the sequence constructed from $q$ as above (in the definition of $\hat{\pi}$) and write 
$\langle \tilde n(i) \mid i<\tilde m\rangle$ for the analogous sequence 
constructed from $\tilde q$.
By the requirement in \eqref{e.def.D} that all trees have at least one splitting node, and since trees can only be extended at their maximal nodes, we have $\stem(t_j) = \stem(\tilde t_j)$ for all $j<l$. 
By the requirement in \eqref{e.def.D} that $p\wedge\stem(t_j)$ is strictly shorter than $p$ for all $j<l$, 
$\langle n(i) \mid i<m\rangle$ is an initial segment of $\langle \tilde n(i) \mid i<\tilde m\rangle$.
Hence the increasing enumeration of $l \setminus \{n(i) \mid i<m\}$ is an initial segment of 
the increasing enumeration of $\tilde l \setminus \{\tilde n(i) \mid  i<\tilde m\}$.

It follows that $l' \leq \tilde l'$ and  $\tilde t'_i$ end-extends $t'_i$ for $0<i<l$.
Since $q\in D$ and thus $|p\wedge \stem(t_{n(m-1)})| = |p|-1$ by \eqref{e.def.D}, it follows that
$p \subseteq \stem(t_{\tilde n(j)})$ whenever $m<j<\tilde m$. 
Thus $\tilde t'_0$ end-extends $t'_0$ by the definition in \eqref{e.t'_0}. 

\medskip 

To see that $\pi$ is a projection, let $q = \big(\langle t_j \mid j < l\rangle,p\big) \in D$ be given; write $q' =\langle t'_n \mid n < l'\rangle$ for $\pi(q)$.
Further suppose we are given $q''=\langle t''_n \mid n < l''\rangle$ in $\TT^\omega$ with $q'' \leq q'=\pi(q)$. 
We aim to find $\tilde q = \big(\langle \tilde t_j \mid j < \tilde l\rangle,\tilde p\big)$ in $D$ with $\tilde q \leq q$ and $\pi(\tilde q) \leq q''$.

Let $\tilde m$ be maximal so that $\pi_{t''_0}(x{\res} \tilde m)$ is defined and let 
\[
\tilde{p} = \pi_{t''_0}(x{\res} \tilde m).
\]
Write $\langle s_i\mid i < \tilde m\rangle$ for the enumeration in order of increasing length of the set of $s \in t''_0$ such that $s \perp \tilde{p}$ but each proper initial segment of $s$ is an initial segment of $\tilde{p}$. 

\medskip

We now take some precautions to ensure that $\tilde q$ as defined below will be an element of $D$ and that $\pi(\tilde q)=q''$. 

Without loss of generality, by replacing each tree $t''_i$ for $i<l''$ by an end-extension if necessary,
we can assume that each tree $(t''_0)_{s_i}$ for $i<\tilde m$ as well as each tree $t''_j$ for $0<j<l''$ has at least one splitting node; this is automatically true only for $i<m$ and $j<l'$.
(Recall here that $(t''_0)_{s_i}=\{s \in t''_0 \mid s \subseteq s_i \lor s_i \subseteq s\}$  as in  Section~\ref{sec:dpideals}, p.~\ref{T_u}.) 

Likewise, 
replacing $t''_0$ by an end-extension if necessary,
we can assume that $\tilde{p}$ is has greater length than $\tilde{p}\wedge\stem(t''_j)$ for each $j<l''$.

Finally, 
again replacing $t''_0$ by an end-extension 
and increasing $\tilde{m}$ if necessary, we can assume that $|s_{\tilde m-1}|=|\tilde{p}|$, or equivalently,
\begin{equation}\label{e.last-tree}
|\tilde{p} \wedge s_{\tilde m-1}|=|\tilde{p}|-1. 
\end{equation}
Making these assumptions will suffice to see that $\tilde q$ as defined below is in $D$.

\medskip

Note that $p$ is an initial segment of $\tilde{p}$, since $q'' \leq q' = \pi(q)$ (whence $t''_0$ end-extends $t'_0$) and by the definition of $\hat{\pi}$. 
We further have $p\wedge \stem(t_{n(i)})\subseteq s_i \subseteq \stem(t_{n(i)})$ for each $i<m$ by the definition of $\langle n(i)\mid i<m\rangle$ 
from $q$ as above. 
Moreover $p\subseteq s_j$ for any $j$ with $m\leq j<\tilde m$,  since 
$|s_{m-1}|=|p|$ or equivalently, $|p \wedge \stem(t_{n(m-1)})|=|p|-1$. 
The latter holds since $q\in D$.

We now define $\tilde{q}$. 
With $\langle\bar n(i)\mid i<l-m\rangle$ defined from $q$ as above, let for $j<l$ 
\[
\tilde t_j = \begin{cases} (t''_0)_{s_i} &\text{if $j = n(i)$ for some $i<m$,}\\
                                      t''_{i+1}  &\text{if $j=\bar n(i)$ for some $i<l-m$.}\\
\end{cases}
\]
Let $\tilde l = l'' + \tilde m-1$ and define 
\begin{tikzpicture}[remember picture,overlay]
\draw[<-]
([shift={(7pt,-5pt)}]pic cs:s) |- ([shift={(-10pt,-15pt)}]pic cs:s) 
  node[anchor=east] {$\scriptstyle l\text{th position}$};  
\draw[<-] 
([shift={(2pt,-5pt)}]pic cs:t) |- ([shift={(-5pt,-23pt)}]pic cs:t) 
  node[anchor=east] {$\scriptstyle (l+\tilde m-m)\text{th position}$};  
\draw[<-] 
([shift={(2pt,-5pt)}]pic cs:last) |- ([shift={(5pt,-15pt)}]pic cs:last) 
  node[anchor=west] {$\scriptstyle (\tilde l-1)\text{th position}$};  
\end{tikzpicture}
\begin{gather*}
\tilde q = \big(\langle \tilde t_0, \hdots, \tilde t_{l-1}, \overbrace{\tikzmark{s} (t''_0)_{s_m}, \hdots, (t''_0)_{s_{\tilde m-1}}}^{\text{if $\tilde m>m$}}, \overbrace{\tikzmark{t}t''_{l'}, \hdots, \tikzmark{last}t''_{l''-1}}^{\text{if $l''>l'$}}\rangle, \tilde{p}\big).\\[8pt]
\end{gather*}
The precautions taken above ensure that $\tilde q \in D$.  
To see that $\pi(\tilde q)\leq q''$, observe that by the definition of $\hat{\pi}$ and \eqref{e.last-tree}, the sequence of $\tilde{n}(i)$ associated with $\tilde{q}$ as above has length $\tilde{m}$; 
we denote it by $\langle \tilde{n}(i)\mid i<\tilde{m}\rangle$. 
Moreover  
\[
\tilde{n}(i) = \begin{cases}
n(i) &\text{if $i<m$,}\\
l-m+i &\text{if $m\leq i < \tilde m$.}
\end{cases}
\]  
It is then a matter of straightforward computation to check that $\pi(\tilde q) = q''$. 
\end{proof}
It remains to verify \eqref{e.iso.cohen} and \eqref{e.iso.trees}. 
Suppose $g \times h$ is $\TT^\omega\times\Cohen$-generic. Let $T_n = \dot T^g_n$ and let $c = \bigcup h$.
By the definition of $\hat{\pi}$, letting 
 $g'=\hat{\pi}(g\times h)$ and writing $\langle T'_n \mid n \ <\omega\rangle$ for the sequence of trees given as
\[
\langle T'_n \mid n \ <\omega\rangle = \hat{\pi}\big(\langle T_n \mid n<\omega\rangle,c\big),
\]
we have $T'_n = \dot T^{g'}_n$ for each $n<\omega$. 
Also by the definition of $\hat{\pi}$, in particular by \eqref{e.t'_0},
\[
c = 
\bigcup_{k\in\omega} \pi_{T'_0}(x{\upharpoonright}k)=  \sigma^{g'}_x= \sigma^{g'} 
\]
proving \eqref{e.iso.cohen}. 
Moreover
\[
T'_0 = \bigcup_{i<\omega} T_{n(i)}
\]
with $n(i)$ as above (in the definition of $\hat{\pi}$) and therefore by construction
\[
[T'_0] \setminus \{c\}= \bigcup_{i<\omega} [T_{n(i)}], 
\]
proving \eqref{e.iso.trees}.
\end{proof}
This completes the proof of the theorem, as we argued after 
Claim ~\ref{cl.isomorphism}. 
\end{proof} 


The next results use the following principle. 
Let 
\emph{internal projective Cohen absoluteness} ($\mathsf{IA}^\CC_{\mathrm{proj}}$) denote the statement that $H_{\omega_1}^{M[g]}\prec H_{\omega_1}$ holds for all sufficiently large regular cardinals $\theta$, countable elementary submodels $M\prec H_{\theta}$ and Cohen generic filters $g$ over $M$ in $V$. 

We will only need the first part of the next lemma; the second part is
an easy observation. Recall that $\PD$ denotes the axiom of projective
determinacy.\footnote{See \cite[Chapter 33]{MR1940513}. } 


\begin{lemma}~\label{internal projective Cohen absoluteness} 
\begin{enumerate-(1)} 
\item 
$\PD$ implies $\mathsf{IA}^\CC_{\mathrm{proj}}$. 
\item 
If $\mathsf{IA}^\CC_{\mathrm{proj}}$ holds, then all projective set have the property of Baire. 
\end{enumerate-(1)} 
\end{lemma} 
\begin{proof} 
(1) 
Take a ${\bf \Sigma}^1_{2n}$-universal $\Sigma^1_{2n}$ subset $A_n$ 
of $({}^\omega2)^2$ for each $n\geq1$. 
If the $\Sigma^1_{2n}$-formula $\varphi_n(x,y)$ defining $A_n$ is chosen in a reasonable way, then for any $\Sigma^1_{2n}$-formula $\psi(x)$ with (hidden) real parameters, there is some $y_{\psi}\in{}^\omega2$ with $\forall x\ (\psi(x)\Longleftrightarrow \varphi_n(x,y_{\psi}))$. 
Moreover, this holds in all transitive models of $\ZFC^-$ and the map $\psi\mapsto y_\psi$ is absolute between such models. 
We assume that in any transitive model of $\ZFC^-$, $A_n$ denotes the set defined by $\varphi_n(x,y)$. 


Let $M \prec H_\theta$ be a countable elementary submodel for some
large enough regular cardinal $\theta$.
By \cite[Corollary 6C.4]{MR2526093} and $\PD$, there is a $\Sigma^1_{2n}$-scale $\sigma_n=\langle \leq^m_n\mid m\in\omega\rangle$ on $A_n$ for each $n\geq1$. 
Let $r^m_n(x)$ denote the rank of $x\in A_n$ with respect to $\leq^m_n$. 
Moreover, recall that the \emph{tree of the scale} $\sigma_n$\footnote{The tree of a scale is defined in the discussion before \cite[Theorem 8G.10]{MR2526093}.} is defined as 
$$T_{\sigma_n}=\{(x{\upharpoonright}m,(r^0_n(x),...,r^{m-1}_n(x)))\mid x\in A_n,\ m\in\omega\}$$ 
and $p[T_{\sigma_n}]=A_n$. 
Since  $\ZFC^-+\PD$ is sufficient to prove the existence of such
scales, let $T_{\sigma_n}$ denote the tree defined via
$\varphi_n(x,y)$ in any transitive model of this theory containing
$\pow(\omega)$ as an element. Let $g \in V$ be Cohen generic over $M$.

\begin{claim} 
$M[g]\models A_n= p[T_{\sigma_n}]$. 
\end{claim} 
\begin{proof} 
$p[T_{\sigma_n}]$ has a projective definition via the above
definitions of $T_{\sigma_n}$ and $A_n$.
Thus the claim holds by $1$-step projective Cohen absoluteness in $M$
\cite[Lemma 2]{Woodin82}. 
\end{proof} 

\begin{claim} 
$T_{\sigma_n}^M=T_{\sigma_n}^{M[g]}$. 
\end{claim} 
\begin{proof} 
By $\PD$ any projective pre-wellorder $E$ is \emph{thin}, i.e. there is no perfect set of reals that are pairwise inequivalent with respect to $E$. 
By $1$-step projective Cohen absoluteness \cite[Lemma 2]{Woodin82}, \cite[Lemma 3.18]{MR3226056} and $\PD$, Cohen forcing does not add new equivalence classes to $E$. 
Equality now follows from \cite[Theorem 5.15]{MR3226056}. 
The inclusion $T_\sigma^M\subseteq T_\sigma^{M[g]}$ holds since the rank function $r_n$ is upwards absolute from $M$ to $M[g]$ by the previous statement for $E=\leq_n$. 
The converse inclusion is proved in \cite[Claim 5.16]{MR3226056} from $\PD$. 
\end{proof} 


We now show that $\mathsf{IA}^\CC_{\mathrm{proj}}$ holds. 
Assume that $\psi(x)$ is a $\Sigma^1_{2n}$-formula and $x\in M[g]$. 
Note that the equivalences $\psi(x)\Longleftrightarrow \varphi_n(x,y_\psi) \Longleftrightarrow (x,y_\psi)\in A_n \Longleftrightarrow (x,y_\psi)\in p[T_{\sigma_n}]$ hold in $V$ and in $M[g]$ by the first claim. 
It remains to show that $M[g]\models (x,y_\psi)\in p[T_{\sigma_n}]$ if
and only if $(x,y_\psi)\in p[T_{\sigma_n}]$ holds in $V$. 

We claim that $T_{\sigma_n}^{M[g]}=T_{\sigma_n}\cap M[g]$. 
To see this, note that $T_{\sigma_n}^M=T_{\sigma_n}\cap M$ since $M\prec H_\theta$. 
Using the second claim and the fact that $\Ord^M=\Ord^{M[g]}$, we obtain $T_{\sigma_n}^{M[g]}=T_{\sigma_n}^M=T_{\sigma_n}\cap M=T_{\sigma_n}\cap M[g]$. 
By absoluteness of wellfoundedness, we have $M[g]\models (x,y_\psi)\in p[T_{\sigma_n}]\Longleftrightarrow (x,y_\psi)\in p[T_{\sigma_n}^{M[g]}] \Longleftrightarrow (x,y_\psi)\in p[T_{\sigma_n}]$.


(2) Since $M$ is countable, the set of Cohen reals over $M$ in $V$ is comeager. Therefore, this claim holds by the argument for the Baire property for definable sets in Solovay's model. 
\end{proof} 

The above argument for the non-existence of a selector with Borel values can now be used for the following results. 
We fix codes for ${\bf \Sigma}^1_n$ sets via ${\bf \Sigma}^1_n$-universal $\Sigma^1_n$ sets. 

%

\begin{theorem} \label{no selector from internal absoluteness} 
Assuming $\mathsf{IA}^\CC_{\mathrm{proj}}$, there is no Baire measurable selector for $\I$ with projective values. 
\end{theorem} 
\begin{proof} 
This is proved similarly to Theorem \ref{no selector for ideal of countable sets}. We only indicate the two necessary changes.  
In the proof of the first claim, $\mathsf{IA}^\CC_{\mathrm{proj}}$ implies that the (projective) statement $[\dot{T}_n^g]\cap B_{F(\dot{C}^g)}=[\dot{T}_n^g]\cap B_{\dot{F}^g(\dot{C}^g)}=\emptyset$ is absolute between $M[g]$ and $V$. 
The second change is at the very end of the proof. Here projective absoluteness between $M[g\times h]$ and $V$ by $\mathsf{IA}^\CC_{\mathrm{proj}}$ guarantees that $x=\sigma^G\in B_{F(\dot{C}^g)}$ and thus $B_{F(\dot{C}^g)}\setminus B_{\dot{C}^g}$ contains every Cohen real over $M[g]$ in $V$. As before, this contradicts 
the fact that $F$ is a selector for $\I$. 
\end{proof}

Note that by Lemma \ref{internal projective Cohen absoluteness}, $\PD$
is sufficient for the previous result.  Thus $\PD$ implies that there
is no projective selector for $\I$. 
In particular, there is no selector $\Borel({}^\omega \baseset)\to\Proj({}^\omega \baseset)$ for $=_\I$ 
that is induced by a universally Baire measurable function (assuming $\mathsf{PD}$).

Note that some assumption beyond
$\ZFC$ is necessary to prove this statement. For instance, in $L$
there are projective selectors for all projectively coded ideals $\I$,
since there is a projective, in fact a $\Sigma^1_2$, wellorder of the
reals.

The previous arguments can also be used to show that it is consistent with $\ZF$ that there is no selector at all for $\I$. 

Suppose that $\kappa$ is an uncountable cardinal and $G$ is a $\PP$-generic filter over $V$ for 
$\PP=\Add(\omega,\kappa)$ or $\PP=\operatorname{Col}(\omega,{<}\kappa)$. 
We call $V^\star= \bigcup_{\alpha<\kappa} V[G{\upharpoonright}\alpha]$ a \emph{symmetric extension} for $\PP$. 
Note that $V^\star=HOD^{V[G]}_{V^\star}$ (the class of sets which are hereditarily ordinal definable in $V[G]$ with parameters from $V^\star$) by homogeneity and for cardinals $\kappa$ of uncountable cofinality, $\RR^{V^\star}=\RR^{V[G]}$. 


\begin{theorem}\label{theorem models without selector} 
There are no selectors for $\I$ in the following models of $\ZF$. 
\begin{enumerate-(a)} 
\item 
Symmetric extensions $V^\star$ for $\Add(\omega,\kappa)$ and $\operatorname{Col}(\omega,{<}\kappa)$, where $\kappa$ is any uncountable cardinal. 
\item 
Solovay's model. 
\item 
$L(\RR)$ assuming $\AD^{L(\RR)}$. 
\end{enumerate-(a)} 
\end{theorem} 
\begin{proof} 
The proof of the first claim is similar to that of to Theorem \ref{no selector for ideal of countable sets}. 

Suppose that $F$ is a selector for $\I$ in $V^\star$. 
It is definable from an element $x_0$ of $V[G{\upharpoonright}\alpha]$ for some $\alpha<\kappa$. 
We can further assume that $x_0\in V$.  
Hence $F$ is continuous on the set of Cohen reals over $V$ in $V^\star$; let $\dot{F}$ be a name for this function. 

We follow the proof of Theorem \ref{no selector for ideal of countable sets} but work with $V$ instead of $M$. 
The first claim in the proof of Theorem \ref{no selector for ideal of
  countable sets} is replaced by the following claim. 

We will write $V^\star$ for a $\PP$-name for $V^\star$ to keep the notation simple. 

\begin{claim} 
$1 \Vdash^V_{\TT^\omega\times \PP} V^\star\models [\dot{T}_n]\cap B_{\dot{F}(\dot{C})}\neq\emptyset$ for all $n\in\omega$. 
\end{claim} 
\begin{proof} 
Assume towards a contradiction that $(p,q) \Vdash^V_{\TT^\omega\times \PP} V^\star\models [\dot{T}_n]\cap B_{\dot{F}(\dot{C})}=\emptyset$ for some $p\in\TT^\omega$, $q\in \PP$ and $n\in\omega$. We can assume $q=1$ by homogeneity. 
Let $g\times h$ be a $\TT^\omega\times \PP$-generic filter over $V$ with $p\in g$ whose symmetric model is $V^\star$. 
Then  $[\dot{T}_n^g]\cap B_{F(\dot{C}^g)}=[\dot{T}_n^g]\cap B_{\dot{F}^g(\dot{C}^g)}=\emptyset$. But  this contradicts the assumption that $F$ is a selector with respect to $\I$. 
\end{proof} 

There is a $\TT^\omega$-name $\sigma$ in $V$ that is forced to be an element of $[\dot{T}_n]\cap B_{\dot{F}(\dot{C})}\cap V^\star$ by the previous claim. The next steps of the proof are as before. 

In the end of the proof, we rearrange $V[g\times h]$ as a $\TT^\omega\times\Add(\omega,1)$-generic extension $V[G\times H]$ with $\sigma^G=x$ as before. 
Then $\sigma^G\in B_{F(\dot{C}^g)}$ by the choice of $\sigma$. 
Thus $B_{F(\dot{C}^g)}\setminus B_{\dot{C}^g}$ contains every Cohen real over $V[g]$ in $V^\star$. 
Since there are uncountably many Cohen reals over $V[g]$ in $V^\star$, this contradicts the assumption that $F$ is a selector. 

The second claim holds since Solovay's model is obtained via a symmetric extension for  $\operatorname{Col}(\omega,{<}\kappa)$. Note that this model is also a $\Add(\omega,\omega_1)$-generic extension of an intermediate model.  

The last claim follows from \cite[Theorem 0.1]{schindlersteel} and the
first claim. By this theorem and $\AD^{L(\RR)}$,
$L(\RR)^V=L(\RR)^{V^\star}$ for some symmetric extension $V^\star$ for
$\operatorname{Col}(\omega,{<}\kappa)$ for some $\kappa$ over some ground model $N$ which is an inner model of some generic extension of $V$. 
\end{proof}

\np

\subsection{Density points via convergence} \label{section convergence density} 

In this section we discuss the notion of density point introduced in \cite{Poreda85}.
We show that this notion does not satisfy the analogue of Lebesgue's density theorem for any of the tree forcings listed in the next section, except Cohen and random forcing. 

Lebesgue's density theorem was generalized to the $\sigma$-ideals of
meager sets on Polish metric spaces in \cite[Theorem 2]{Poreda85}.  To
this end, a notion of density points for ideals was introduced.  
This notion is
based on the following well-known measure theoretic lemma. 

\begin{lemma} \label{convergence in measure} 
Suppose that
$(\XX,d,\mu)$ is a metric measure space, $f$ is a Borel-measurable
function and $\vec{f}=\langle f_n \mid n\in\omega\rangle$ is a
sequence of Borel-measurable functions from $(\XX,d)$ to $\RR$.
The following statements are equivalent: 
\begin{enumerate-(a)} 
\item 
$f_n\rightarrow f$ converges in measure, i.e. for all $\epsilon>0$ we have 
$$\lim_{n\rightarrow\infty} \mu(\{x\in \XX\mid |f_n(x)-f(x)|\geq\epsilon\})=0.$$ 
\item \label{condition b}
Every subsequence of $\vec{f}$ has a further 
subsequence that converges pointwise $\mu$-almost everywhere. 
\end{enumerate-(a)} 
\end{lemma} 

Condition \ref{condition b} is suitable for generalizations to other
ideals, since it only mentions the ideal of null sets, but not the
measure itself.  

\begin{definition} Let $\I$ be a $\sigma$-ideal on ${}^\omega \baseset$ and $A\subseteq {}^\omega \baseset$. 
\begin{enumerate-(1)}
\item 
Given $x\in {}^\omega\baseset$, let $f_n$ denote the characteristic function of ${\sigma_{x{\upharpoonright}n}^{-1}(A\cap N_{x{\upharpoonright}n})}$ for each $n\in\omega$. 
Define $x$ to be an \emph{$\I$-convergence density point} of $A$
if every subsequence of $\vec{f}=\langle f_n\mid n\in\omega\rangle$ 
has a further subsequence that converges $\I$-almost everywhere (i.e. on a set $A$ with ${}^\omega \baseset\setminus A \in \I$) to the constant function on ${}^\omega \baseset$ with value $1$. 
\item 
By the \emph{$\I$-convergence density property} we mean the statement that for any $B\in\Borel({}^\omega \baseset)$ and the set $C$ of  $\I$-convergence density points of $B$, we have $B\triangle C \in \I$. 
\end{enumerate-(1)}
\end{definition}

We first consider tree forcings on ${}^\omega\omega$. 
For $s\in {}^{<\omega}\omega$ and $f_0,\dots,f_m\in {}^\omega\omega$, let
$$T_{s,f_0,\dots,f_m}=\{t\in C_s\mid \forall i\leq m\ \forall n\geq |s|\ t(n)\neq f_i(n)\}$$ 
(cf.\ Definition \ref{def:listoftreeforcings} \ref{definition eventually different forcing} below) and 
 let $f\in{}^\omega\omega$ denote the constant function with value $0$. 

\begin{lemma} 
If $N_{\langle 0\rangle}$ and $[T_{\emptyset,f}]$ are $\I$-positive, then the $\I$-convergence density property fails. 
In particular, this holds for $\I_{\Hechler}$, $\I_\EE$, $\I_{\Miller}$  
and $\I_{\Laver[A]}$ if $A$ contains all cofinite sets. 
\end{lemma} 
\begin{proof} 
We claim that no $x\in{}^\omega\omega$ is an $\I$-convergence density point of $[T_{\emptyset,f}]$. 
We have $\sigma_t^{-1}([T_{\emptyset,f}])=[T_{\emptyset,f}]$ if $t\in T_{\emptyset,f}$ and $\sigma_t^{-1}([T_{\emptyset,f}])=\emptyset$ otherwise. 
Hence $\sigma_t^{-1}([T_{\emptyset,f}])\cap N_{\langle 0\rangle}=\emptyset$ for all $t\in {}^{<\omega}\omega$. 
Thus $x$ is not an $\I$-convergence density point of $[T_{\emptyset,f}]$. 
\end{proof} 

We now turn to tree forcings on ${}^\omega 2$. 
For $s \in {}^{<\omega}2$ and $N\in[\omega]^\omega$
let
$$T_{s,N}=\{ t \in {}^{<\omega}2 \mid t\in C_s\ \&\ \forall n\ t(n)=1\Rightarrow n\in N\} $$ 
(cf.\  Definition \ref{def:listoftreeforcings} \ref{definition Mathias} below). 

\begin{lemma} 
Suppose that $N_{\langle 1\rangle}$ and $[T_{\emptyset,N}]$ are $\I$-positive for some infinite set $N$. 
Then the $\I$-convergence density property fails. 
In particular, this holds for $\Mathias[A]$ and $\Silver[A]$ for any subset $A$ of $\pow(\omega)$ that contains at least one infinite set. 
\end{lemma} 
\begin{proof} 
It is sufficient to show that no $x\in [T_{\emptyset,N}]$ is an $\I$-convergence density point of $[T_{\emptyset,N}]$. 
We have $\sigma_{x{\upharpoonright}n}^{-1}([T_{\emptyset,N}])\cap N_{\langle 1\rangle}=\emptyset$ for all $n\notin N$. 
Since $N_{\langle 1\rangle}$ is $\I$-positive, any infinite strictly increasing sequence in $N$ witnesses that $x$ is not an $\I$-convergence density point. 
\end{proof} 

The next lemma takes care of the remaining tree forcings. 

\begin{lemma} 
Suppose that $\I$ is an ideal on ${}^\omega 2$ with $\I_\Sacks\subseteq \I\subseteq \I_{\Mathias}$. 
Then the $\I$-convergence density property fails. 
\end{lemma} 
\begin{proof} 
Let 
\begin{align*}
T = \{t \in {}^{<\omega}2 \mid \forall i \in \omega\ (t(i)=1\Rightarrow \exists j\in\omega\  (i=2^j+1)) \}. 
\end{align*} 
It is sufficient to show that no $x\in {}^\omega 2$ is an $\I$-convergence density point of $[T]$. 
Otherwise there is a strictly increasing sequence $\vec{n}=\langle n_i \mid i <\omega\rangle$ such that 
$$A = \{y\in {}^\omega 2\mid \exists i\ \forall j\geq i\ y\in \sigma_{x{\upharpoonright} n_j}^{-1} ([T])\}$$ 
is co-countable. 
Let 
$$A_{i,j} = \sigma_{x{\upharpoonright} n_i}^{-1}([T]) \cap \sigma_{x{\upharpoonright} n_j}^{-1}([T])$$ 
for $i<j$ in $\omega$. 
Since $A \subseteq \bigcup_{i,j\in\omega,\ i<j } A_{i, j}$ is uncountable, $A_{i,j}$ is uncountable for some $i<j$. 
Let $a<_{\mathrm{lex}}b<_{\mathrm{lex}}c$ be elements of $A_{i,j}$. 
We denote the longest common initial segment of $d,e\in {}^{\leq\omega}2$ by $d{\wedge} e$. 
We can assume without loss of generality that $a{\wedge} b=s$, $a{\wedge} c=b{\wedge} c=t$ and $s\subsetneq t$. 
Then $s$ and $t$ are splitting nodes in $\sigma_{x{\upharpoonright}n_i}^{-1}(T)$ and $\sigma_{x{\upharpoonright} n_j}^{-1}(T)$. 
Hence $s+n_i$, $s+n_j$, $t+n_i$, $t+n_j$ are of the form $2^k+1$. 
Since $s\neq t$, $n_j-n_i$ can be written in the form $2^k-2^l$ in two different ways. 
But this contradicts the easy fact that $k,l$ are uniquely determined by $2^k-2^l$. 
\end{proof} 

Let $\PP_{E_0}$ denote \emph{$E_0$-forcing} \cite[Section 4.7.1]{MR2391923} and $\mathbb{W}$ an appropriate representation of \emph{Willowtree forcing} \cite[Section 1.1]{MR1354935}. 
Since $\Mathias \subseteq \Silver\subseteq \mathbb{W}\subseteq
\PP_{E_0} \subseteq \Sacks$, the previous result holds for the ideals
$\I_\PP$ associated to these forcings as well.

\section{
A list of tree forcings} \label{section list of tree forcings} 

We review definitions of some tree forcings for the reader's convenience. 
If $N\subseteq \omega$ and $m\in\omega$, we write $m+N=\{m+n \mid n\in N\}$. 
Moreover, a subset $A$ of $\pow(\omega)$ is called \emph{shift invariant} if $N\in A \Longleftrightarrow m+N\in A$ for all $m\in \omega$. 
  
\begin{definition} \label{def:listoftreeforcings} 
Assume that $A$ is a subset of $\pow(\omega)$. 
\begin{enumerate-(1)} 
\item\label{definition random} 
\emph{Random forcing} is the collection of perfect subtrees $T$ of ${}^{<\omega}2$ with $\mu([T_s])>0$ for all $s\in T$ with $\stem(T)\sqsubseteq s$. 

\item 
\emph{Cohen forcing $\Cohen$} is collection of \emph{cones} $C_s=\{ t \in {}^{<\omega}\omega \st s \sqsubseteq t \text{ or }t\sqsubseteq s\}$ for $s \in {}^{<\omega}\omega$. 

\item 
\emph{Sacks forcing $\Sacks$} is the collection of all perfect subtrees of ${}^{<\omega}2$. 

\item 
\emph{Miller forcing $\Miller$} is the collection of \emph{superperfect} subtrees $T$ of ${}^{<\omega}\omega$. This means that above every node in $T$ there is some infinitely splitting node $t$ in $T$, i.e. $t$ has infinitely many direct successors. 

\item \label{definition Hechler} 
\emph{Hechler forcing $\Hechler$} is the collection of trees 
$$T_{s,f}=\{t\in C_s\mid \forall n\geq |s|\ t(n)\geq f(n)\}$$ 
for $s\in {}^{<\omega}\omega$ and $f\in {}^\omega\omega$. 
  
\item \label{definition eventually different forcing} 
\label{def:listoftreeforcings e} 
\emph{Eventually different forcing $\Eventual$} is the collection of trees 
$$T_{s,f_0,\dots,f_m}=\{t\in C_s\mid \forall i\leq m\ \forall n\geq |s|\ t(n)\neq f_i(n)\}$$ 
for $s\in {}^{<\omega}\omega$ and $f_0,\dots,f_m\in {}^\omega\omega$. 

\item 
\emph{$A$-Laver forcing $\Laver[A]$} is the collection of subtrees $T$ of ${}^{<\omega}\omega$ such that for every $t\in T$ with $\stem(T)\sqsubseteq t$, the set of immediate successors of $t$ in $T$ is an element of $A$. \emph{Laver forcing} is $\Laver[F]$ for the Fr\'echet filter $F$ of cofinite sets. 
  
\item \label{definition Mathias} 
\emph{$A$-Mathias forcing} $\Mathias[A]$ is the collection of trees 
$$T_{s,N}=\{ t \in {}^{<\omega}2 \mid t\in C_s\ \&\ \forall n\ t(n)=1\Rightarrow n\in N\} $$
for $s \in {}^{<\omega}2$ and $N\in A$. 
\emph{Mathias forcing} $\Mathias$ is $\Mathias[A]$ for $A = [\omega]^\omega$. 
  
\item 
\emph{$A$-Silver forcing} $\Silver[A]$ is the collection of trees 
$$T_f=\{t\in {}^{<\omega}2\mid \forall n\in \dom(t)\cap \dom(f) \ f(n)\leq t(n)\},$$ 
where $\dom(f)=\omega\setminus N$ for some $N\in A$. 
Silver forcing $\Silver$ is $\Silver[A]$ for $A = [\omega]^\omega$. 
\end{enumerate-(1)}
\end{definition}

All of these satisfy the above condition for collections of trees $\PP$ that $T_s\in\PP$ for all $T\in\PP$ and $s\in T$. 
Moreover, random forcing, Cohen forcing, Sacks forcing, Hechler
forcing, eventually different forcing, and $A$-Laver forcing are \emph{shift invariant} in the sense that for all $T$ and $s\in {}^{<\omega}2$, $T\in\PP\Longleftrightarrow \sigma_s(T)\in\PP$. 
If $A$ is shift invariant, then $\Mathias[A]$ and $\Silver[A]$ are also shift invariant. 


Note that Cohen forcing, Hechler forcing, eventually different forcing, Laver forcing, and Silver forcing are topological, while random forcing, Sacks forcing, and Miller forcing are not. 

\section{An explicit construction of density points} \label{section explicit construction of density points} 

In this section, we show how to explicitly construct density points of a closed set $C$ of positive measure. 

In fact, by the next result there is an algorithm that takes as input a list of data from $C$ and outputs a perfect tree (level by level) all of whose branches are density points. (By choosing e.g. the leftmost branch, we can approximate a single density point with arbitrary precision.) 
The input is a tree $T$ together 
with weights $w_t=\frac{\mu([T]\cap N_t)}{\mu(N_t)}>0$ for all $t\in T$; 
we call this a \emph{weighted tree}. 
The weights are given as inputs with arbitrary precision. 



\begin{theorem} \label{construction of density points} 
There is a partial computable function that takes as input any pair $(T,q)$, where $T$ is a weighted tree and $q\in\QQ\cap(0,1)$, and produces a perfect tree $S\subseteq T$ with 
\begin{enumerate-(a)} 
\item \label{construction of density points 1} 
$[S]\subseteq D_\mu([T])$ 
and 
\item \label{construction of density points 2} 
$\mu([S]) \geq q\mu([T])$. 
\end{enumerate-(a)} 
\end{theorem} 


Since 
$\mu([T_s])$ is right-c.e.\footnote{See \cite[Definition 1.8.14]{MR2548883}.} in the oracle $T$ for all $s\in T$, 
one has the following immediate consequence. 

\begin{corollary} 
For any tree $T$ with $\mu([T])>0$ and $q\in \QQ\cap (0,1)$, there is a $\Delta^0_2(T)$-definable perfect tree $S$ 
such that \ref{construction of density points 1} and \ref{construction of density points 2} of Theorem \ref{construction of density points} hold. 
Moreover, there is a $\Delta^0_2(T)$-coded $F_\sigma$ set $A\subseteq D_\mu([T])$ with $\mu(A)=\mu([T])$. 
\end{corollary}

Note that for all strongly linked collections of trees $\PP$ (see Definition \ref{definition strongly linked}) listed in Section \ref{section list of tree forcings} and all $T\in \PP$, $S=T$ already satisfies the conditions in 
Theorem \ref{construction of density points}. 
For these collections, any $T\in \PP$ has the property that for all $x\in [T]$, there are infinitely many $n\in\omega$ such that there is some $S \leq T$ with $x \in [S]$ and $\stem(S) = x{\upharpoonright}n$. 
Hence all elements of $[T]$ are density points of $[T]$ by the proof
of Lemma \ref{every branch in T in a density point} below.

Theorem \ref{construction of density points} will follow from the next lemmas. 
To state them, we fix the following notation: Let $C=[T]$, 
$L_{t,i}=\Lev_{|t|+i}(T_t)$ be the level of $T_t$ at height $|t|+i$ 
and write 
$$w_{t,i}=\frac{|L_{t,i}|}{2^i}=\mu(N_t)^{-1} \frac{|L_{t,i}|}{2^{|t|+i}}$$ 
for all $t\in {}^{<\omega}2$ and $i\in\omega$. 
This is the \emph{relative size} of levels of $T$ above $t$. 
The next result shows that these values converge to the relative measure at $t$. 

\begin{lemma} \label{measure as limit of relative level sizes} 
$\lim_{i\rightarrow\infty}w_{t,i}=w_t$ for all $t\in {}^{<\omega}2$. 
\end{lemma} 
\begin{proof} 
We have $w_t\leq\lim_{i\rightarrow\infty}w_{t,i}$, since $C\cap N_t\subseteq \bigcup_{u\in L_{t,i}} N_u$ and hence $w_t \leq w_{t,i}$ for all $i\in\omega$. 
To prove that $\lim_{i\rightarrow\infty}w_{t,i}\leq w_t$, suppose that $\epsilon>0$ is given. Let $U$ be an open set with $C\cap N_t\subseteq U$ and $\mu(U)<\mu(C\cap N_t)+\epsilon \cdot \mu(N_t)$. By compactness of $C$, we can assume that $U$ is a finite union of basic open sets. We can thus write $U=\bigcup_{j\leq n} N_{s_j}$ for some $\vec{s}=\langle s_j\mid j\leq n\rangle$ that consists of pairwise incompatible sequences $s_j$ of the same length $|t|+i$. 
Since $C\cap N_t\subseteq U$, we have 
$$w_{t,i}=\mu(N_t)^{-1} \frac{|L_{t,i}|}{2^{|t|+i}}\leq \frac{\mu(\bigcup_{j\leq n} N_{s_j})}{\mu(N_t)} = \frac{\mu(U)}{\mu(N_t)}.$$ 
Hence 
$w_{t,i}-w_t \leq \frac{\mu(U)-\mu(C\cap N_t)}{\mu(N_t)}<\epsilon$ by the previous inequality and the definition of $w_t$. 
\end{proof} 

For any  $t\in {}^{<\omega}2$ and $i\in\omega$, let 
$$r_{t,i}=\inf\{c\in (0,1) \mid \frac{|\{u\in L_{t,i}\mid w_u \geq c\}|}{|L_{t,i}|}\geq c\}$$ 
denote the ratio of nodes on level $|t|+i$ above $t$ with large weight. 

\begin{lemma} \label{large fraction with large weight} 
$\liminf_{i\rightarrow\infty} r_{t,i}=1$ for all $t\in {}^{<\omega}2$. 
\end{lemma} 
\begin{proof} 
Let 
$b=w_t$ and assume that $c\in(0,1)$ is given. 
Since $b> b c+b(1-c)c$, there is some $\epsilon>0$ with $b> (b+\epsilon) c+(b+\epsilon)(1-c)c.$
By Lemma \ref{measure as limit of relative level sizes}, we can take $i\in\omega$ to be sufficiently large such that $w_{t,i}\leq b+\epsilon$. Moreover, let $\alpha$ denote the fraction of nodes $u\in L_{t,i}$ with weight $w_u\geq c$. Then 
$$b= w_t = 2^{-i}\sum_{u\in L_{t,i}} w_u \leq  w_{t,i} \alpha+ w_{t,i} (1-\alpha)c.$$ 
We claim that $\alpha\geq c$. Otherwise $\alpha<c$ and 
$b\leq w_{t,i} \alpha+w_{t,i} (1-\alpha)c\leq w_{t,i} c+w_{t,i} (1-c)c$.  
Since $w_{t,i}\leq b+\epsilon$, we obtain  
$b\leq (b+\epsilon) c+(b+\epsilon) (1-c)c$, contradicting the definition of $\epsilon$. 
\end{proof} 

We need the following notion to ensure that weights converge to $1$ along branches of the tree constructed below. 
We say that $v\sqsupseteq t$ is \emph{$(t,a)$-good} 
if $w_u\geq a$ for
all $u$ with $t\sqsubseteq u\sqsubseteq v$; otherwise it is called \emph{$(t,a)$-bad}. 
Let 
$$s_{t,a,i}=\frac{|\{u\in L_{t,i}\mid \text{$u$ is $(t,a)$-good}\}|}{|L_{t,i}|}$$ 
denote the fraction of $(t,a)$-good nodes on level $|t|+i$ above $t$. 

We fix a computable function $f\colon \QQ\cap (0,1)\rightarrow \QQ\cap (0,1)$ such that $\frac{1-b}{b(b-a)}< \frac{1-a}{a}$ 
holds for all 
$a,b\in \QQ\cap(0,1)$ with $b>f(a)$. 

\begin{lemma} \label{large fraction of good nodes} 
If $a\in\QQ\cap (0,1)$ and $w_t= b> f(a)$, then $\liminf_{i\rightarrow\infty} s_{t,a,i}\geq a$. 
\end{lemma} 
\begin{proof} 
Since $b>f(a)$, there is some $c\in(0,1)$ with 
$1-b < \frac{1-a}{a} b(b-a) c$. 
Let $i$ be sufficiently large such that the fraction of $v\in L_{t,i}$ with $w_v\geq c$ is at least $b$ by Lemma \ref{large fraction with large weight}. 
Then the fraction of nodes $v\in L_{t,i}$ with $w_v<c$ is at most $1-b$ and their number at most $2^i w_{t,i}(1-b)$. 

Let $A$ the set of $(t,a)$-bad nodes in $L_{t,i}$, $U=\bigcup_{v\in A} N_v$ and $\alpha=\frac{|A|}{|L_{t,i}|}$. 
We aim to show that $\alpha\leq 1-a$. 
The number of $(t,a)$-bad nodes in $L_{t,i}$ is $2^i w_{t,i}\alpha$. 
Since all of these except at most $2^i w_{t,i}(1-b)$ have weight at least $c$, we have  
$$2^{|t|+i}\mu(C\cap U)= \sum_{v\in A}w_v \geq (2^i w_{t,i}\alpha- 2^i w_{t,i}(1-b)) c.$$ 

\begin{claim} 
$1-b \geq \frac{1-a}{a} w_{t,i}(\alpha-(1-b))c$. 
\end{claim} 
\begin{proof} 
For each $v\in A$, 
take some $u_v$ with
$t\sqsubseteq u_v\sqsubseteq v$ and $w_{u_v}<a$ by the definition of $(t,a)$-bad. 
In particular $\frac{1-w_{u_v}}{w_{u_v}} > \frac{1-a}{a}$.

Let $B=\{ u_v\mid v\in A\}$ and $B^\star$ the set of $\sqsubseteq$-minimal
elements 
of $B$. 
For each $v\in A$ and $u=u_v$ 
$$\mu(N_u\setminus C)=2^{-|u|}(1-w_u)>2^{-|u|} w_u \frac{1-a}{a}=\frac{1-a}{a} \mu(C\cap N_u). $$ 
Since the sets $N_u$ for $u\in B^\star$ are pairwise disjoint, the previous inequality implies 
$$\mu(U\setminus C) = \sum_{u\in B^\star} \mu(N_u\setminus C)  > \frac{1-a}{a} \sum_{u\in B^\star} \mu(C\cap N_u) = \frac{1-a}{a}\mu(C\cap U).$$ 
By this inequality and the one before the claim, we have $\mu(N_t\setminus C) \geq \mu(U\setminus C)>\frac{1-a}{a}\mu(C\cap U)\geq 2^{-|t|} \frac{1-a}{a} (w_{t,i}\alpha-w_{t,i}(1-b)) c$. 
Since $\frac{\mu(N_t \setminus C)}{2^{-|t|}}=\frac{\mu(N_t \setminus C)}{\mu(N_t)}=1-w_t=1-b$ and $\mu(N_t)=2^{-|t|}$, the claim follows. 
\end{proof} 

It is sufficient to show that $\alpha\leq 1-a$. Otherwise by the previous claim and since $b=w_t\leq w_{t,i}$ 
\begin{equation*} 
\begin{split}
1-b & \geq \frac{1-a}{a} w_{t,i}(\alpha-(1-b)) c \\ 
& > \frac{1-a}{a} w_{t,i}((1-a)-(1-b)) c \\ 
& = \frac{1-a}{a} w_{t,i}(b-a) c \\ 
& \geq \frac{1-a}{a} b(b-a) c. 
\end{split} 
\end{equation*} 
But this contradicts the choice of $c$. 
\end{proof} 

%
%

\begin{proof}[Proof of Theorem \ref{construction of density points}] 
Let $\vec{a}=\langle a_i\mid i\in\omega\rangle$ be a computable sequence in $\QQ\cap (0,1)$ with $\prod_{i\in\omega}a_i^2>q$. 
Using Lemmas \ref{large fraction with large weight} and \ref{large fraction of good nodes}, we will inductively construct a strictly increasing sequence $\vec{n}=\langle n_i\mid i\in\omega\rangle$ and 
sets $S_i\subseteq \Lev_{n_i}(T)$ with $n_0=0$ and $S_0=\{\emptyset\}$ by induction on $i\in\omega$. 
The sets $S_i$ are compatible levels of a tree in the sense that each $t\in S_i$ has an extension $u\in S_{i+1}$ and conversely, each $u\in S_{i+1}$ extends some $t\in S_i$. 
We further let $T^{(i)}=\{t\in T\mid \exists u\in S_i\ (t\sqsubseteq u\vee u\sqsubseteq t)\}$ denote the subtree of $T$ induced by $S_i$. 

We will maintain during the induction that 
(a) $u$ is $(t,a_i)$-good for all $t\in S_{i+1}$ and $u\in S_{i+2}$ 
and 
(b) $\frac{\mu([T^{(i+1)}])}{\mu([T^{(i)}])}\geq a_i^2 $. 

\medskip 
\noindent 
{\bf Description.} 
We now describe the construction. 
We simultaneously construct auxiliary numbers $a'_i, a''_i, b_i\in\QQ$ and $j_i\in\omega$ with $a_i<a'_i<a''_i< b_i<1$ and $b_i>f(a''_i)$ for all $i\geq 1$. 
It is not hard to see that all steps are effective. 

Let $n_0=0$ and $S_0=\Lev_{n_0}(T)=\Lev_0(T)$. 

For $i=1$, 
we first choose some $a'_1\in \QQ$ with $a_1<a'_1<1$ and $b_1\in\QQ\cap (0,1)$ with $b_1>f(a'_1)$. 
By Lemma \ref{large fraction with large weight} applied to $t=\emptyset$, there is some $j_1$ with 
$r_{\emptyset,j_1}>b_1$. 
Let $n_1=j_1$ and $S_1$ a subset of $\Lev_{n_1}(T)$ with $w_u>b_1$ for all $u\in S_1$ and $\frac{|S_1|}{|\Lev_{n_1}(T)|}>b_1$. 
We can further take $b_1$ to be sufficiently large such that $\frac{\mu([T^{(1)}])}{\mu([T^{(0)}])}> a_0^2$ by $r_{\emptyset,j_1}>b_1$ and the definition of $r_{\emptyset,j_1}$. 

Fix $i\geq1$ and assume that step $i$ is completed. 
First take some $a'_{i+1}, a''_{i+1}\in\QQ$ with $a_{i+1}<a'_{i+1}<a''_{i+1}<1$ and $1-a'_{i+1}<a'_i-a_i$. 
Let $b_{i+1}\in\QQ\cap(0,1)$ with $b_{i+1}> f(a''_{i+1}), a''_{i+1}, a_i$. 
By Lemma \ref{large fraction with large weight}, 
there is some $j_{i+1}$ 
such that $r_{t,j_{i+1}}> b_{i+1}$ for all $t\in S_i$. 
Since $b_i>f(a''_i)$ by the inductive hypothesis and $a'_i<a''_i$, we can take $j_{i+1}$ to satisfy $s_{t,a''_i,j_{i+1}}> a'_i$ for all $t\in S_i$ by Lemma \ref{large fraction of good nodes}. 
Let $n_{i+1}=n_i+j_{i+1}$. 
By definition of $j_{i+1}$, there is a subset $S_{i+1}$ of $\Lev_{n_{i+1}}(T)$ such that for all $u\in S_{i+1}$, we have $w_u> b_{i+1}$, there is a (unique) $t\in S_i$ with $t\sqsubseteq u$, $u$ is $(t,a'_i)$-good and 
$\frac{|L_{t,j_{i+1}}\setminus S_{i+1}|}{|L_{t,j_{i+1}}|}< (1-a'_i)+(1-b_{i+1})$. 

\medskip 
\noindent 
{\bf Verification.}
We show that the algorithm computes the required tree. 
Clearly condition (a) is maintained in the construction. 
The next claim shows (b). 
\begin{claim} 
$\frac{\mu([T^{(i+1)}])}{\mu([T^{(i)}])} \geq a_i^2$ for all $i\in\omega$. 
\end{claim} 
\begin{proof} 
This is clear for $i=0$. 
Let $i\geq1$ and fix any $t\in S_i$. 
Since $a'_{i+1}<b_{i+1}$ and by the definition of $S_{i+1}$, we have 
$\frac{|L_{t,j_{i+1}}\setminus S_{i+1}|}{|L_{t,j_{i+1}}|}\leq (1-a'_i)+(1-b_{i+1}) \leq (1-a'_i)+(1-a'_{i+1}) \leq 1-a_i$. 
Hence 
$\frac{|L_{t,j_{i+1}}\cap S_{i+1}|}{|L_{t,j_{i+1}}|}\geq a_i$. 

Since each $u\in  S_{i+1}$ has weight at least $b_{i+1}$, we have 
$c:=\frac{\mu([T^{(i+1)}]\cap N_t)}{\mu([T^{(i)}]\cap N_t)}\geq c':= a_i b_{i+1}$. 
By the definition of $T^{(i+1)}$ from $S_{i+1}$, $d:=\frac{\mu(([T^{(i)}]\cap N_t)\setminus [T^{(i+1)}])}{\mu([T^{(i)}]\cap N_t)}\leq d':= 1- a_i$. 
Moreover $c+d=1$. 
Since $c\geq c'$, $\frac{c+d'}{c}=1+\frac{d'}{c}\leq 1+\frac{d'}{c'}=\frac{c'+d'}{c'}$. 
Since $d\leq d'$ and by the last inequality $\frac{c}{c+d}\geq \frac{c}{c+d'}\geq \frac{c'}{c'+d'}$. 
Therefore  
$c =\frac{c}{c+d}\geq \frac{c'}{c'+d'}= \frac{a_i b_{i+1}}{a_i b_{i+1}+ (1-a_i)} \geq a_i b_{i+1} \geq a_i^2$. 
Since this inequality holds for all $t\in S_i$, we have 
$\frac{\mu([T^{(i+1)}])}{\mu([T^{(i)}])} \geq a_i^2$. 
\end{proof} 

To see that conditions (a) and (b) are sufficient, let $S$ be the unique perfect subtree of ${}^{<\omega} 2$ with $\Lev_{n_i}(S)=S_i$ for all $i\in\omega$. 
This tree can be computed level by level via the algorithm above. 
We have $\lim_{i\rightarrow \infty} a_i=1$ by the definition of $\vec{a}$. 
Thus (a) implies that all elements of $[S]$ are density points of $[T]$. 
Moreover, $\mu([S])=\inf_{i\in\omega} \mu([T^{(i)}])\geq \prod_{i\in\omega} a_i^2 \mu([T])<q \mu([T])$ by (b) as required. 
\end{proof} 

The previous result provides a finitized proof of Lebesgue's density theorem for Lebesgue measure on Cantor space, since any Borel set can be approximated in measure by closed subsets. 
To see this, note that trivially $D(A)\cap D({}^\omega\baseset\setminus A)=\emptyset$ for any subset $A$ of ${}^\omega\baseset$. 
Thus it is sufficient to show that for any Borel set $A$ and any $\epsilon>0$, there is a closed subset $C$ of $A$ with $\mu(A\setminus C)<\epsilon$ consisting of density points of $A$; the density property for $A$ follows by applying this property to both $A$ and its complement. 
To see that this property holds, take a closed subset $B$ of $A$ with $\mu(A\setminus B)<\frac{\epsilon}{2}$. By Theorem \ref{construction of density points}, there is a closed subset $C$ of $B$ with $\mu(B\setminus C)<\frac{\epsilon}{2}$ that consists of density points of $B$ and therefore also of $A$. 
Since $\mu(A\setminus C)<\epsilon$, $C$ is as required. 

Note that the algorithm also produces lower bounds for weights along branches of $S$.

\section{Open problems} \label{section open problems} 

We end with some open problems. 
Our main goal is to prove the equivalence of the properties discussed in Section \ref{subsection - conclusion} for many other ideals. 







\begin{question} 
Are the shift density property, the existence of a simply definable selector and the ccc equivalent for all simply definable $\sigma$-ideals? 
\end{question} 

This is of interest for the ideals studied in \cite{MR2391923}, for instance 
for the $K_\sigma$-ideal. 
For the latter, we suggest to generalize the proof idea of Theorem \ref{no selector for ideal of countable sets} to answer the next question. 

\begin{question} 
Is there a Baire measurable selector with Borel values for the $K_\sigma$-ideal? 
\end{question} 

A promising related problem is whether there is a relation between the shift density property and the condition that the collection of Borel sets modulo $\I$ carries a natural Polish metric. 





\medskip 

Our proof of the density property for strongly linked tree ideals is based on the fact from Section \ref{section P-measurable sets} that $\PP$-measurable sets form a $\sigma$-ideal if $\PP$ has the $\omega_1$-covering property. 
Can the latter assumption be omitted? 
In this case, the next question would also have a positive answer. 

\begin{question} 
Are all Borel sets $\PP$-measurable for all tree forcings $\PP$? 
\end{question} 

Note that any counterexample $\PP$ collapses $\omega_1$ if we assume $\CH$; in this case $\PP$ preserves $\omega_1$ if and only if it has the $\omega_1$-covering property by a similar argument as in the proof of Lemma \ref{preservation of omega1}.





\medskip 

In Section \ref{section: no selector}, we proved from $\mathsf{PD}$ that there is no selector $\Borel({}^\omega \baseset)\to\Proj({}^\omega \baseset)$ for $=_\I$ that is induced by a universally Baire measurable function on the codes (Theorem \ref{intro - no projective selector}) 
To this end, we introduced the principle $\mathrm{IA}^\CC_{\mathrm{proj}}$ of generic absoluteness. 
This follows from $\PD$ by Lemma \ref{internal projective Cohen absoluteness}. Our results leave its precise consistency strength open. 



\begin{question} 
What is the consistency strength of $\mathrm{IA}^\CC_{\mathrm{proj}}$? 
\end{question} 
Recent unpublished work of the first and second-listed authors answers this by showing that the consistency strength of $\mathrm{IA}^\CC_{\mathrm{proj}}$ is that of $\ZFC$ (see~\cite{mueller-schlicht-preprint}).

\bibliographystyle{alpha}
\nocite{*}
  \bibliography{lebesguedensity}

\end{document}